\documentclass[11pt]{amsart}
\usepackage{a4wide,enumerate,enumitem,amsmath,color}
\usepackage[utf8]{inputenc}
\usepackage[scale = .8]{geometry}
\usepackage{amssymb}
\usepackage{amsthm}
\usepackage{graphicx}
\usepackage{subcaption}
\usepackage[noadjust]{cite}
\usepackage{bigints}
\usepackage{breqn}
\usepackage{tikz}
\usepackage{mathtools}
\usepackage{hyperref}
\usepackage{cite}
\usepackage{dsfont}
\usepackage {float}

\newtheorem{theorem}{Theorem}

\newtheorem{lemma}{Lemma}
\newtheorem{proposition}{Proposition}
\newtheorem{remark}{Remark}
\newtheorem{definition}{Definition}

\newcommand{\K}{\mathcal{K}}
\newcommand{\dd}{\mathrm{d}}
\newcommand{\pa}{\partial}

\newcommand{\R}{\mathbb{R}}
\newcommand{\Z}{\mathbb{Z}}

\newcommand{\I}{\mathcal{I}_N}

\newcommand{\eps}{\varepsilon}

\begin{document}

\title[Convergence of a scheme for a nonlocal eikonal equation]{Convergence of a scheme for a one dimensional nonlocal and nonlinear eikonal equation}

\author[D. Al Zareef]{Diana Al Zareef}
\address{Universit\'e de Technologie de Compi\`egne, LMAC, 60200 Compi\`egne, France}
\email{diana.al-zareef@utc.fr}

\author[A. El Hajj]{Ahmad El Hajj}
\address{Universit\'e de Technologie de Compi\`egne, LMAC, 60200 Compi\`egne, France}
\email{elhajjah@utc.fr}

\author[A. Zurek]{Antoine Zurek}
\address{Universit\'e de Technologie de Compi\`egne, LMAC, 60200 Compi\`egne, France}
\email{antoine.zurek@utc.fr}

\date{\today}

\thanks{}

\begin{abstract}In this work, we study a one-dimensional nonlocal and nonlinear eikonal equation without sign restriction on its spatial gradient. The equation is characterized by weak regularity assumptions on both the nonlocal velocity field and the initial data. We derive a periodic version of this model and propose a semi-explicit (IMEX) scheme for its numerical approximation. We prove that the scheme preserves a discrete gradient entropy estimate and establish its convergence in the viscosity sense. Finally, we present numerical results illustrating the behavior of the model and the performance of the proposed scheme.

\bigskip
		

\medskip
		
\end{abstract}

%\paragraph{Keywords:}
%\keywords{}

%\paragraph{AMS classification:}
%\subjclass[2000]{}

\maketitle
\tableofcontents
%\tableofcontents
%%%%%%%%%%%%%%%%%%%%%%%%%%%%%%%%%%%%%%%%%%%%%%%%%%%%%%%%%%%%%%%%%%%%%%%%%%%%%%%
\section{Introduction}
\subsection{Physical motivation and overview of the main results}Understanding the physical dynamics of materials by examining the movement of atoms within them has long been of great interest to scientists. Indeed, in 1934,  Taylor, Orowan, and Polanyi were the first to introduce the concept of dislocations in crystal structures.

A dislocation is defined as a linear defect in a crystalline material, particularly in metallic alloys. Their pioneering work demonstrated that these defects play a key role in explaining plastic deformation at the microscopic scale (see~\cite{19, hull2011introduction} for a physical description of dislocations). A dislocation is characterized by its Burgers vector, which quantifies both the magnitude and the direction of the lattice distortion. There exist several types of dislocations, one of which is the edge dislocation. In this case, an extra half-plane of atoms is inserted into the crystal, and the Burgers vector is perpendicular to the dislocation line. It indicates the direction of motion of the dislocation under an applied stress.

In this work, we are interested in a particular geometry in which the dislocation lines are parallel and move within the same two-dimensional $(xy)$ plane embedded in a three-dimensional elastic crystal. By taking a cross-sectional view perpendicular to the dislocation lines, the model reduces from a three-dimensional setting to a two-dimensional one. Furthermore, we assume that the distribution of dislocations depends only on one spatial variable. Under this assumption, the model further reduces to a one-dimensional problem. This leads to the one-dimensional model for dislocation dynamics considered in the present work. Accordingly, the model, with the unknown real-valued function $u$ describing the plastic deformation induced by dislocations, is given by
\begin{align}\label{3eq:eikonal}
    \begin{cases}
  \partial_t u(x,t) = \left[\left( \K(\cdot) \star u(\cdot,t)\right)(x) + a(t)\right]
  \left|\partial_x u(x,t)\right|& \mbox{in } \R \times (0, T), \\
  u(x,0) = u_0(x) & \mbox{in } \R,
\end{cases}
\end{align}
where $T>0$, $\partial_t u$ and $\partial_x u$ denote respectively the time and space derivatives of $u$, and $\K$ is a kernel obtained from the solution of the equations of linear elasticity. Here, the dislocations move with the following nonlocal velocity \[ \K(\cdot) \star u(\cdot,t)(x) + a(t), \] where the nonlocal term \[ \K(\cdot) \star u(\cdot,t)(x) = \int_{\R} \K(x-z)\,u(z,t)\,dz \] represents the self-force generated by the interactions among dislocations, while the term $a(t)$ corresponds to the external force induced by the applied shear strain field.
 It is defined in this way to study the time evolution of the dislocation densities. For a comprehensive overview of the modeling framework, we refer the reader to the work of O. Alvarez et al. ~\cite{alvarez2006dislocation}.\\
Before stating our assumptions, let us recall the definition of the so-called Zygmund space:
\begin{equation}\label{3defLlogL}
    L\, \log \,L(\R) = \left\{ f \in L^1_{{\rm loc}}(\R) \, : \, \int_{\R} |f|\ln( e + |f|) \, \dd x <  \infty \right\},
\end{equation}
which is a Banach space with the norm 
\begin{align*}
\|f\|_{L\,\log \,L(\R)} = \inf \left\{ \mu > 0 \, : \, \int_{\R} \frac{|f|}{\mu} \,\ln \left( e + \frac{|f|}{\mu} \right) \, \dd x \leq 1 \right\}.
\end{align*}
This functional setting naturally arises in our analysis, in particular in the assumptions on the initial data. 
In the sequel, we consider convenient assumptions on $\K$, $a$, and $u_0$ ensuring the existence of at least one continuous viscosity solution to~\eqref{3eq:eikonal} (see for instance~\cite{1}). These assumptions are as follows:
\smallskip
\begin{itemize}
	\item [\textbf{(H0)}] \label{3a} {\bf{Shear strain field}}: the function $a$ belongs to $C(0,T)$.
\item[\textbf{(H1)}] {\bf{Initial data}}: $u_0 \in L^\infty(\R)$ with $\partial_x u_0 \in L^1(\R) \cap L\,\log\,L(\R)$.

\item[\textbf{(H2)}] {\bf Kernel}: $\K \in L^1(\R)$ is an even function such that
\begin{equation}\label{3fourier transform is negative}
\displaystyle \int_{\R} \mathcal{K}(x) \dd x = 0\quad \mbox{and}\quad 
\mathcal{F}(\mathcal{K})(\xi) \leq 0, \quad \forall \xi \in \R. 
\end{equation}
\end{itemize}

Our aim in this paper is to construct a numerical scheme that approximates problem \eqref{3eq:eikonal}. Following~\cite{1}, one can show that the solutions to~\eqref{3eq:eikonal} satisfy the following gradient entropy estimate for all $t\in (0,T)$:
\begin{align}\label{31.entv}
    \int_{\R} f\left(|\pa_x u (x, t)|\right) \, \dd x - \int_0^t \int_{\R} \left(\K(\cdot)\star \pa_x u(\cdot,s)(x) \right) \, \pa_x u(x,s) \, \dd x \dd s \leq C,
\end{align}
where
\begin{equation}\label{3def:entropydensity}
    0 \leq f(x) = \begin{cases}
        x \ln(x) + \frac{1}{e} & \quad\mbox{if } x\geq 1/e,\\
        0  & \quad \mbox{if } 0 \leq x \leq  1/e.
    \end{cases}          
\end{equation}
This inequality is a key estimate for establishing the existence of solutions to~\eqref{3eq:eikonal}. For this reason, when we design a numerical scheme for~\eqref{3eq:eikonal}, the scheme must preserve a similar gradient entropy estimate at the discrete level. This property is essential to prove convergence. However, due to the lack of regularity in the velocity field, we cannot work directly with the initial model. To solve this difficulty, we need to derive a regularized version of~\eqref{3eq:eikonal} which admits a gradient entropy estimate similar to~\eqref{31.entv}. Hence, this paper has two main objectives:
\begin{itemize}
\item Derive a reliable periodic regularized model from~\eqref{3eq:eikonal} by regularizing the singular kernel $\K$ thanks to the F\'ejer kernel;
\item Design a convergent numerical scheme for this periodic regularized model.
\end{itemize}
The numerical scheme designed for the regularized model is introduced in Section~\ref{3numerical scheme} . It is a one step implicit/explicit (IMEX) Euler in time and finite difference in space scheme. We prove that the scheme is monotone, stable, and consistent, and that it preserves a discrete gradient entropy estimate. Then, using the theory of viscosity solutions and the convergence result for monotone schemes established by Barles and Souganidis in~\cite{barles_souganidis_1991}, we conclude that the discrete solution converges to the viscosity solution of the continuous problem~\eqref{3eq:eikonal}.

\subsection{State of the art and main novelties of this work}
Let us briefly review some analytical and numerical studies related to nonlocal Hamilton--Jacobi and transport equations, particularly in connection with dislocation dynamics.

In \cite{ZHIZ25}, Al Zareef et al. studied the same equation as in the present work, which can be seen as a special case of our setting. By assuming monotonicity of the solution, they were able to remove the absolute value in the eikonal term, thereby reducing the model to a transport equation. Based on this reformulation, they designed a convergent semi-explicit upwind scheme and established a discrete gradient entropy estimate, proving convergence toward a weak solution in the sense of distributions. In the present work, we remove this monotonicity assumption. As a consequence, the absolute value cannot be eliminated, and the equation retains its genuinely nonlinear eikonal structure. This prevents convergence in the distributional sense, due to the lack of continuity when passing to the limit through the absolute value of the gradient.

More generally, several works have addressed nonlocal Hamilton--Jacobi equations under structural assumptions ensuring regularity or monotonicity. For instance, in the homogenization setting of one-dimensional dislocation models, Ghorbel et al. \cite{ghorbel2008numerical} studied a nonlocal Hamilton--Jacobi equation under periodicity and Lipschitz regularity assumptions on the external stress, together with structural conditions on the kernel ensuring coercivity-type inequalities. Their numerical scheme is monotone under suitable CFL conditions.

Another contribution was the work of Ghorbel and Monneau in \cite{ghorbel2010well}, where they studied a one-dimensional nonlocal transport equation modeling parallel dislocations. They proved the global well-posedness of viscosity solutions and designed a monotone finite difference scheme together with an error estimate. Their analysis relies on strong structural assumptions, including Lipschitz continuity of the initial data, strong regularity of the velocity, and the nonnegativity of the velocity. However, in our setting, we neither assume Lipschitz continuity of the initial data nor impose any sign condition on the velocity.

In a related framework, Alvarez et al. \cite{alvarez2006convergence, alvarez2006convergent} established convergence results for first-order monotone schemes for nonlocal eikonal equations within the Crandall--Lions framework. Their analysis assumes that the velocity is bounded and Lipschitz continuous with respect to all variables, and that the initial data is globally Lipschitz continuous. These hypotheses ensure the validity of the comparison principle and stability of the numerical scheme. 

We also mention the work of Souganidis \cite{souganidis1985approximation}, which provides convergence results for general approximation schemes of viscosity solutions for first-order Hamilton–Jacobi equations, together with explicit error estimates. The framework developed in \cite{souganidis1985approximation} is closely connected to the viscosity convergence analysis used in the present paper. However, unlike the local equations considered there, our model involves a nonlocal convolution-dependent velocity field and requires additional entropy-based arguments to handle the nonlinear eikonal structure without sign assumptions on the gradient.

In addition, we mention the contribution of El Hajj and Forcadel \cite{MR2373180}, who studied a $(2 \times 2)$ quasi-monotone system in a one-dimensional setting, proposed an explicit numerical scheme, and proved its convergence to the Lipschitz continuous solution using comparison principle property. Their work was inspired by Alvarez et al. \cite{alvarez2006convergent}, and established a Crandall–Lions–type rate of convergence \cite{1350}. Later, El Hajj and Oussaily \cite{MR4734513} considered an implicit numerical scheme for the same system as in \cite{MR2373180}, assuming that the initial data are non-decreasing functions and being in $H^1$-regularity. Their result relies on an $L^2$ gradient estimate that was proved in the continuous setting.

For non-conservative $(d \times d)$ hyperbolic systems with $d \geq 2$, Monasse and Monneau \cite{MoMo14} established the convergence of a semi-explicit scheme for diagonal strictly hyperbolic systems by means of discrete gradient entropy estimates, under structural assumptions including strict hyperbolicity, global Lipschitz continuity of the velocity, and monotone, bounded initial data. Their approach relies on this type of estimate, originally derived in the continuous framework \cite{MR3104078}. Their analysis is carried out within the framework of vanishing viscosity solutions, first introduced by Bianchini and Bressan \cite{MR2150387}.

Nonetheless, we mention the work of Al Zohbi et al. \cite{alzohbi_elhajj_jazar_2022}, who studied a $(d \times d)$ system of nonlinear eikonal equations in one space dimension. In their setting, the velocity is assumed to be bounded, continuous, Lipschitz continuous with respect to the spatial variable, and monotone with respect to the unknown. The initial data is additionally assumed to be bounded and continuous. They introduced a semi-explicit numerical scheme that satisfies a $BV$ estimate and proved its convergence in the viscosity sense. Their analysis first treats discontinuous viscosity solutions in the $BV$ setting, and, under additional assumptions ensuring a comparison principle, yields convergence toward the unique continuous viscosity solution.

As noticed, all the aforementioned works rely on some combination of key structural assumptions, such as Lipschitz continuity of the initial data, regularity and boundedness of the velocity, monotonicity properties, or a fixed sign condition.

In contrast, the model studied in the present work operates under significantly weaker regularity assumptions. The nonlocal eikonal equation considered here involves a velocity that is neither monotone nor of fixed sign, and does not satisfy any Lipschitz regularity condition. Furthermore, the initial data is not assumed to be Lipschitz continuous. As a consequence, the comparison principle is no longer available, and standard convergence frameworks cannot be applied. In addition, convergence in the distributional sense cannot be established due to the lack of continuity when passing to the limit through the absolute value of the gradient. Despite these difficulties, we prove, using the discrete gradient entropy estimate, the convergence of our numerical scheme toward a viscosity solution of the nonlocal eikonal equation. The theoretical results are further supported by numerical simulations, which illustrate the behavior of the scheme.

\subsection{Derivation of a periodic regularized model} \label{3derivation of a periodic model}
For a fixed positive number \(P\), we introduce a periodic approximation of the initial data in order to study the bulk behavior of the material away from boundary effects. More precisely, we define
\begin{equation}\label{31.initperiod}
u_0^P(x):=u_0(x)-L^P x,
\qquad x\in[-P,P],
\end{equation}
where
\begin{equation}\label{3def LP}
L^P=\frac{u_0(P)-u_0(-P)}{2P}.
\end{equation}
The function \(u_0^P\) is then extended by \(2P\)-periodicity over $\R$. The linear term \(L^P x\) is introduced precisely to ensure that
\[
u_0^P(P)=u_0^P(-P),
\]
which guarantees the periodicity of the initial data. This construction allows us to describe the evolution of dislocations in a periodic domain, avoiding boundary effects, thereby capturing the effective bulk dynamics of the material. In the periodic setting, the equation is naturally posed on the torus
\begin{align}\label{3defI}
I_P:=\mathbb{R}/(2P\mathbb{Z}).
\end{align}

As already explained, we need to derive a periodic counterpart of \eqref{31.entv}. To this end, we carefully modify the kernel $\K$. Following the approach introduced in \cite{ZHIZ25}, we first periodize the kernel \(\K\). More precisely, we define the \(2P\)-periodic function \(\K^P\) by
\[
\K^P(x)=\K(x),
\qquad x\in[-P,P],
\]
and extend it periodically over \(\mathbb{R}\). We then consider the periodic problem
\begin{align}\label{32.periodic}
\begin{cases}
\partial_t u^P(x,t)
=
\left[
\left(\K^P(\cdot)\star u^P(\cdot,t)\right)(x)
+a(t)
\right]
|\partial_x u^P(x,t)|
& \text{in } I_P\times(0,T),
\\
u^P(x,0)=u_0^P(x)
& \text{in } I_P.
\end{cases}
\end{align}

However, the direct use of the periodized kernel \(\K^P\) is not sufficient to establish the desired discrete gradient entropy estimate. We refer the reader to \cite{ZHIZ25} for a detailed explanation of this issue. For this reason, and following the approach developed in \cite{ZHIZ25}, we introduce a suitable regularization based on the Cesàro means of the Fourier expansion. To this end, let \(M>1\) be a fixed integer and define the Fejér kernel
\begin{equation}\label{3Fejer kernel}
F_M(x)
=
\frac{1}{M}
\sum_{\ell=0}^{M-1} D_\ell(x),
\end{equation}
where \(D_\ell\) denotes the Dirichlet kernel
\[
D_\ell(x)
=
\sum_{|m|\le \ell}
e^{i\pi m x/P},
\qquad \ell=0,\dots,M-1.
\]
Next, we introduce the modified periodic kernel
\begin{align}\label{31.defKP}
\K_M^P(x)
=
\K^P(x)
-
\frac{2F_{2M}(x)}{P}
\int_{|x|\ge P} |\K(x)|\,dx,
\qquad x\in\mathbb{R}.
\end{align}
We then define the Cesàro mean of order \(M\) associated with the Fourier series of \(\K_M^P\):
\begin{align}\label{3cesaro}
\sigma_M^P(x)
:=
\sigma_M(\K_M^P)(x)
=
\frac1M
\sum_{\ell=0}^{M-1}
\sum_{|m|\le \ell}
c_m(\K_M^P)
\,e^{i\pi m x/P}
=
\frac{1}{2P}
\left(F_M\star \K_M^P\right)(x),
\qquad x\in I_P,
\end{align}
where \(c_m(\K_M^P)\) denotes the \(m\)-th Fourier coefficient of \(\K_M^P\), namely
\[
c_m(\K_M^P)
=
\frac{1}{2P}
\int_{I_P}
\K_M^P(x)\,
e^{-i\pi m x/P}\,dx,
\qquad m\in\mathbb{Z}.
\]
Moreover, in \eqref{32.periodic} and \eqref{3cesaro} the symbol $\star$ denotes the convolution between $2P$-periodic functions defined, for any $w$ and $z\in L^1(I_P)$, as
\begin{align}\label{3conv.prod.perio}
    (w\star z)(x) = \int_{I_P} w(x-y) \, z(y) \, \dd y, \quad \mbox{for a.e. }x \in I_P.
\end{align}
Hence, in this paper, we will study the solutions, denoted $u^P_M$, of the following regularized periodic system:
 \begin{align}\label{31.equPaux}
 \begin{cases}
\partial_t u^P_M(x, t) = \left[\left(\sigma_M^P(\cdot) \star u^P_M(\cdot,t)\right)(x) + a(t)\right] \, \left|\partial_x u^P_M(x, t) \right| & (x,t) \in I_P \times (0, T),\\
u^P_M(x,0) = u^P_0(x) & x \in I_P,
\end{cases}
 \end{align}
 where $I_P$ is defined in~\eqref{3defI}. Let us assume that $P$, $M$ and $u^P_0$ satisfy the following assumptions:

 \begin{itemize}
\item[\textbf{(H3)}] Parameters of the periodic model: $P$ is a real positive number and $M > 1$.

\item[\textbf{(H4)}] Periodic initial data: for any $P >0$, the function $u^P_0$ belongs to $L^\infty(I_P)$ with $\pa_x u^P_0 \in L^1(I_P) \cap L\,\log\,L(I_P)$. 
\end{itemize}
It is important to understand how the periodic initial data $u^P_0$ relates to the original initial data $u_0$. The following lemma provides the corresponding uniform estimates.
\begin{lemma}\label{3main_2_lem.init.perio.to.nonperiod}
Let assumption~\textbf{\emph{(H1)}} holds and assume that $P\ge 1$. Then, the function $u^P_0$, given by~\eqref{31.initperiod}, satisfies, uniformly with respect to $P$,
\begin{align*}
    u^P_0 \in L^\infty(I_P), \quad \pa_x u^P_0 \in L^1(I_P)\cap L\,\log \,L(I_P).
\end{align*}    
\end{lemma}

\begin{proof}
Let $x \in [-P,P)$, then, since $P>0$, we have
\begin{align}\label{3estimationLinP}
    \|u^P_0\|_{L^\infty(I_P)}\leq \|u_0\|_{L^\infty(\R)} + \frac12 \left|u_0(P)-u_0(-P)\right|\leq 2\|u_0\|_{L^\infty(\R)},
\end{align}
so that $u^P_0 \in L^\infty(I_P)$ uniformly w.r.t $P$. Moreover, as $P>0$, it holds
\begin{align}\label{3estimationL1P}
    \|\pa_x u^P_0\|_{L^1(I_P)} \leq \|\pa_x u_0\|_{L^1(\R)} + 2\|u_0\|_{L^\infty(\R)}. 
\end{align}
Then, applying inequality~\eqref{3w in llogl} in Lemma~\ref{3Llog L estimate} to the constant function equal to one, we obtain
\begin{align}\label{3estimationLlogLP}
    \|\pa_x u^P_0\|_{L\,\log\,L(I_P)} 
    &\leq  \|\pa_x u_0\|_{L\,\log\,L(I_P)} + \|u_0\|_{L^\infty(\R)} 
    \left(\frac{1}{P}+ 2\ln(1+e^2)+\frac{2}{e}\right), \\
    & \leq \|\pa_x u_0\|_{L\,\log\,L(\R)} + \|
    u_0\|_{L^\infty(\R)} 
    \left(1+ 2\ln(1+e^2)+\frac{2}{e}\right),
\end{align}
where we have used the assumption $P\ge 1$ for the last inequality. 
\end{proof}
Therefore, the initial condition~\eqref{31.initperiod} satisfies assumption~\textbf{(H4)}.

\subsection{Organization of the paper}
The paper is organized as follows. In Section~\ref{3numerical scheme}, we introduce the numerical scheme and state the main results. Section~\ref{3proving_gradient_entropy} is devoted to the analysis of the properties of the proposed scheme and is divided into two subsections. In the first subsection, we prove the well-posedness of the discrete problem via the Banach fixed point theorem and establish discrete total variation and $L^{\infty}$ estimates. In the second subsection, we establish a monotonicity property for the discrete entropy functional, after several technical lemmas on convexity and Fourier coefficients of the periodic kernel. Next, Section \ref{3Convergence of the scheme} is devoted to the proof of convergence towards the viscosity solution of the problem. Finally, we present numerical experiments in Section \ref{3numerical experiment}.

\section{Numerical scheme and main results}\label{3numerical scheme}

\subsection{Numerical scheme} We first introduce some notation for the meshes in time and space. Let $N$ and $N_T$ be positive integers. We define the time step $\Delta t > 0$ by $\Delta t = T/N_T$ and the space step $\Delta x > 0$  by $\Delta x = P/N$. Accordingly, we define the following sequences
\[
t_n =n\Delta t, \quad \forall n \in \{0,\ldots, N_T\},
\]
 and
\[
x_i = i \Delta x, \quad \forall i\in \I := \Z/2N\Z.
\]
Next, we discretize the initial condition $u_0^P$ as
\begin{align}\label{3sch:IC}
u_i^{P, 0} =  u_0^P(x_i), \quad \forall i \in \I.
\end{align}
 Then, for a given vector $u^{P, n} = (u^{P, n} _i)_{i \in \I}$, where $u^{P, n}_i$ approximates $u^P_M(x_i,t_n)$, we define the vector $u^{P, n+1}= (u^{P, n+1} _i)_{i \in \I}$ as the solution of the following system: 
\begin{align}\label{3sch:eq}
  \dfrac{u_i^{P, n+1} - \frac{u^{P, n}_{i+1}+u_i^{P, n}}{2}}{\Delta t} &=   \lambda_i[u^{P, n + 1}] \, \left|\theta^{P, n}_{i+1/2} \right|, \quad \forall i \in \I,
\end{align}
where $ \theta^{P,n}_{i+1/2}$ is the discrete gradient of $u^{P, n}_i$ defined by 
\begin{align}\label{3def:disgrad}
    \theta^{P, n}_{i+1/2} := \dfrac{u^{P, n}_{i+1}-u^{P, n}_i}{\Delta x}, \quad \forall i\in\I.
\end{align}
Finally, we introduce
\begin{equation}\label{3def:lambda}
 \lambda_i[u^{P, n + 1}] = a(t_{n + 1}) + \sum_{j \in \I} \Delta x \, \sigma^P_{M, j} \, u^{P, n + 1}_{i - j}, \quad \forall i \in \I,
 \end{equation}
  with $\sigma_{M,j}^P=\sigma_{M}^P(x_j)$ and where $\sigma_{M}^P$ is the Ces\`aro mean of order $M$ 
 of the Fourier series of $\K^P_M$, introduced previously in \eqref{3cesaro}. 
 \begin{remark}
 It is clear that the vector $u^{P,n}$, for any $1\leq n\leq N_T$, solution to~\eqref{3sch:eq}, depends on $M$. However, when no confusion can occur, we slightly abuse the notation and omit $M$ from its indices.
\end{remark}

 %I can put the same as the remark done in the first article and use this new definition of the kernel.
%where $\sigma^P_{M, j} = \sigma^P_M(x_j)$ with $\sigma^P _M$ is the Cesàro mean of order $M$ of the Fourier series of $\K^P_M$, introduced previously in \eqref{3cesaro}. where, defining $x_{i\pm 1/2} := (i\pm 1/2) \Delta x$ for any $i \in \I$, we set
%\begin{align}\label{3def:diskernel}
%\K_i = \frac{1}{\Delta x} \int_{x_{i-1/2}}^{x_{i+1/2}} \K(x) \, \dd x, \quad \forall i \in \I.
%\end{align}

%%%%%%%%%%%%%%%%%%%%%%%%%%%%%%%%%%%%%%%%%%%%%%%%%%%%%%%%%%%%%%%%%%%%%%%%%%%%%%%
\subsection{Main results} 
Our first main result deals with the well-posedness of  the scheme \eqref{3sch:IC}-\eqref{3sch:eq}, as well as the qualitative properties of its solution.

\begin{theorem}[Existence of solution]\label{3theorem-properties}
    Let assumptions \textbf{\emph{(H0)-(H4)}} hold, and assume that 
    \begin{align}\label{3deltat/deltax}
    \dfrac{\Delta t}{\Delta x} = \frac{1}{4\left(10 \|\K\|_{L^1(\R)} \|u_0\|_{L^{\infty}(\R)} + \|a\|_{L^{\infty}(0, T)}\right)}.
\end{align}
Then, the scheme~\eqref{3sch:IC}-\eqref{3sch:eq} admits a unique solution $u^{P, n} = (u^{P, n} _i)_{i \in \I}$ for any $1\le n \le N_T$, such that   
\begin{align}\label{3ineq:Linfty}
\max_{i \in \I} \left|u^{P,n}_i\right| \leq 2 \|u_0\|_{L^{\infty}(\R)}.
\end{align}
Moreover, this solution satisfies the following discrete total variation estimate
\begin{align}\label{3TV:ineq}
\operatorname{TV}\!\left(u^{P,n}\right) = \sum_{i\in\I} |u^{P,n}_{i+1}-u^{P,n}_i| \leq \|\partial_x u_0\|_{L^1(\R)},
\quad \mbox{for } n \in \{1,\ldots, N_T\}.
\end{align}
Finally, for all $n \in \{1,\ldots,N_T\}$, there exists a positive constant $C = C(\|\K\|_{L^1(\R)}, \|\partial_x u_0\|_{L^1(\R)})$, such that the following discrete gradient entropy estimate
\begin{align}\label{3gradient_entropy_estimate}
\sum_{i \in \I} \Delta x \, f\!\left(\left|\theta^{P, n}_{i+1/2}\right|\right)
\leq 2\sum_{i \in \I} \Delta x 
f\!\left(\left|\theta^{P, 0}_{i+1/2}\right|\right)
 + \ln(2)\,\|\partial_x u_0\|_{L^1(\R)}
+ 2 e^{-1}
+ C\, T,
\end{align}
holds, where we recall definition~\eqref{3def:entropydensity} of $f$.
\end{theorem}
The proof of Theorem~\ref{3theorem-properties} is carried out in Section~\ref{3proving_gradient_entropy} in several steps. First, we establish the existence and uniqueness of a discrete solution to the semi-explicit scheme~\eqref{3sch:IC}--\eqref{3sch:eq} by means of Banach’s fixed point theorem. Then, we derive a series of uniform estimates: the $L^{\infty}$ bound~\eqref{3ineq:Linfty}, the total variation estimate~\eqref{3TV:ineq}, and finally the discrete gradient entropy estimate~\eqref{3gradient_entropy_estimate}.

In our second main result, we prove the convergence of the scheme~\eqref{3sch:IC}--\eqref{3sch:eq}. In order to precisely state this result, we introduce some notation. Let the assumptions of Theorem~\ref{3theorem-properties} hold. Then, for any $\Delta x$ and $\Delta t$ (satisfying~\eqref{3deltat/deltax}), we set $\eps := (\Delta x,\Delta t)$, and define the following function:
\begin{align*}
    u^{P,\eps}_M(x_i,t_n) = u^{P,n}_i, \quad \forall i \in \I, \, 0\leq n \leq N_T,
\end{align*}
where of course $u^{P,n}$ is the unique solution of the scheme~\eqref{3sch:IC}--\eqref{3sch:eq} obtained in Theorem~\ref{3theorem-properties}. We also define, for any $(x, t) \in [x_i, x_{i + 1}] \times [t_n, t_{n+1}]$, its so-called $Q_1$-extension, still denoted $u^{P,\eps}_M$, as 
\begin{multline}\label{3defQ1ext}
u^{P, \varepsilon}_M(x, t) = 
\left( \frac{t - t_n}{\Delta t} \right)
\left[
\left( \frac{x - x_i}{\Delta x} \right) u^{P, n+1}_{i+1}
+ \left( 1 - \frac{x - x_i}{\Delta x} \right) u^{P, n+1}_i
\right] \\
\quad +
\left( 1 - \frac{t - t_n}{\Delta t} \right)
\left[
\left( \frac{x - x_i}{\Delta x} \right) u^{P, n}_{i+1}
+ \left( 1 - \frac{x - x_i}{\Delta x} \right) u^{P, n}_i
\right]. 
\end{multline}
so that $u^{P,\eps}_M(x_i,t_n)=u^{P,n}_i$. In the sequel, we will consider a sequence $\eps_m=(\Delta x_m, \Delta t_m)$ such that $\eps_m \to (0,0)$ as $m \to \infty$, where $\Delta x_m$ and $\Delta t_m$ satisfy~\eqref{3deltat/deltax} for all $m \in \mathbb{N}$. We claim that the sequence $(u^{P,\eps_m}_M)_{m \in \mathbb{N}}$ converges, as $m\to+\infty$, to a solution of~\eqref{31.equPaux} in the viscosity sense. More precisely, we have the following result.

\begin{theorem}[Convergence of the numerical scheme]\label{3theorem-convergence}
Assume the hypotheses of Theorem~\ref{3theorem-properties} hold. For $\eps = (\Delta x, \Delta t)$, with $\Delta x$ and $\Delta t$ satisfying the CFL condition \eqref{3deltat/deltax}, let $u^{P,\varepsilon}_M$ be the corresponding solution of the scheme~\eqref{3sch:IC}--\eqref{3sch:eq}. Then this solution satisfies the following properties:
\begin{equation}\label{3first_L_infty_estimate}
\|u^{P,\varepsilon}_M\|_{L^\infty(I_P\times(0,T))}
\le \|u^P_0\|_{L^\infty(I_P)},
\end{equation}
\begin{equation}\label{3bounded ux}
\|\partial_x u^{P,\varepsilon}_M\|_{L^\infty(0,T;L^1(I_P))}
\le \|\partial_x u_0\|_{L^1(\R)},
\end{equation}
\begin{equation}\label{3time estimate of u}
\|\partial_t u^{P,\varepsilon}_M\|_{L^\infty(0,T;L^1(I_P))}
\le C\, \|\partial_x u_0\|_{L^1(\R)},
\end{equation}
\begin{equation}\label{3bounded in LlogL}
\|\partial_x u^{P,\varepsilon}_M\|_{L^\infty(0,T;L\log L(I_P))} \leq C,
\end{equation}
where $C >0$ is independent of $\varepsilon, P$, and $M$. Moreover, let $\varepsilon_m=(\Delta x_m, \Delta t_m) \to (0,0)$ as $m \to +\infty$, with each $\Delta x_m$ and $\Delta t_m$ satisfying the CFL condition~\eqref{3deltat/deltax}. Then, there exists a function $u^P_M$ such that, up to a subsequence,
\[
u^{P,\eps_m}_M \longrightarrow u^P_M \quad \text{locally uniformly on } I_P \times [0,T].
\]
Furthermore, $u^P_M$ is a viscosity solution of \eqref{31.equPaux} (see Definition~\ref{3def:discont_visc_fixed}), with the following regularity properties:
\begin{equation}\label{3limit-regularity}
\begin{array}{cc}
u^P_M \in L^\infty(I_P\times(0,T)) \cap C([0,T]; L^1(I_P)),
&
\partial_x u^P_M \in L^\infty(0,T; L\log L(I_P)), \\[6pt]
\multicolumn{2}{c}{
\partial_t u^P_M \in L^\infty(0,T; L^1(I_P)).
}
\end{array}
\end{equation}
\end{theorem}

The proof of Theorem~\ref{3theorem-convergence} is carried out in
Section~\ref{3Convergence of the scheme}.
We proceed in two main steps. First, using the uniform estimates
\eqref{3first_L_infty_estimate}--\eqref{3bounded in LlogL},
we establish compactness of the family
$(u^{P,\varepsilon}_M)_\varepsilon$
in $C([0,T];L^1(I_P))\cap L^{\infty}(I_P \times(0, T))$. Second, using the monotonicity and consistency of the numerical scheme,
together with the stability theory for viscosity solutions,
we prove that the limit function $u^P_M$ obtained in the first step is a viscosity solution of
\eqref{31.equPaux}.

\begin{theorem}[Limit $M \to +\infty$ and $P \to +\infty$]\label{3theorem-limit-PM}
Let $u^P_M$ be a viscosity solution of~\eqref{31.equPaux} obtained in Theorem~\ref{3theorem-convergence}. Then, as $M \to +\infty$, there exists a subsequence (still denoted $u^P_M$), and a function $u^P$ such that
\[
u^P_M \longrightarrow u^P \quad \text{locally uniformly on } I_P \times [0,T].
\]
Moreover, $u^P$ is a viscosity solution of~\eqref{32.periodic}
and satisfies the uniform bounds
\[
u^P \in L^\infty(I_P\times[0,T]), 
\quad \partial_x u^P \in L^\infty(0,T;L\log L(I_P)), 
\quad \partial_t u^P \in L^\infty(0,T;L^1(I_P)).
\]
Furthermore, as $P \to + \infty$, there exists a subsequence (still denoted $u^P$) and a function $u$ such that
\[
u^P \longrightarrow u \quad \text{locally uniformly on } \mathbb{R} \times [0,T].
\]
Finally, $u$ is a viscosity solution of the problem \eqref{3eq:eikonal}, with the following regularities
\[
u \in L^\infty(\mathbb{R}\times[0,T]), 
\quad \partial_x u \in L^\infty(0,T;L\log L_{\rm loc}(\mathbb{R})), 
\quad \partial_t u \in L^\infty(0,T;L^1_{\rm loc}(\mathbb{R})).
\]
\end{theorem}

The proof of Theorem~\ref{3theorem-limit-PM} is presented in Section~\ref{3Convergence of the scheme} and proceeds in two steps. In the first step, we let $M \to +\infty$. Using Fejér's theorem, we show that the modified kernel
$\sigma^P_M$ converge to $\mathcal K^P$ in $L^1(I_P)$.
The uniform bounds established in
Theorem~\ref{3theorem-convergence} yield compactness of the family
$(u^P_M)$.
The stability of viscosity solutions under locally uniform convergence
then implies that the limit $u^P$ solves the periodic problem
\eqref{32.periodic}. In the second step, we pass to the limit as $P \to +\infty$.
The uniform estimates inherited from the previous step ensure
compactness on compact subsets of $\mathbb R \times [0,T]$.
Applying again the stability theorem for viscosity solutions,
we conclude that the limit function $u$ is a viscosity solution of
\eqref{3eq:eikonal}.

\section{Study of the numerical scheme}\label{3proving_gradient_entropy}
In this section, we prove Theorem \ref{3theorem-properties}. To this end, we divide the proof into two main steps. In Section~\ref{3well-posedness of the scheme}, we establish the well-posedness of the scheme, while in Section~\ref{3subsec:gradent}, we prove the discrete gradient entropy inequality~\eqref{3gradient_entropy_estimate}.

\subsection{Well-posedness of the scheme}\label{3well-posedness of the scheme}Let us now prove the well-posedness of the scheme thanks to Banach's fixed point theorem. For this, we argue by induction. Thus, let $0\leq n \leq N_T-1$ be fixed, and let $u^{P,n} = (u^{P,n}_i)_{i\in\I}$ be a given vector such that 
\[
\max_{i\in\I}\, |u^{P,n}_i| \leq \alpha := 2\|u_0\|_{L^\infty(\R)},
\]
and
\begin{align*}
\operatorname{TV}\!\left(u^{P,n}\right) \leq \operatorname{TV}\!(u^{P, 0}), 
\quad \mbox{for } n \in \{1,\ldots, N_T\}.
\end{align*}
First, we introduce the following compact set:
\[
\mathcal{U}^\alpha := \{ v = (v_i)_{i \in \I} \mid v_i \in [-\alpha,\alpha] \},
\]
and, for any \( v \in \mathcal{U}^{\alpha} \) and for each \( i \in \I \), we define the mapping

\[
\begin{aligned}
F_i:\quad &\mathcal U^\alpha \longrightarrow \mathbb R,\\
          &v \longrightarrow F_i(v)
           := \frac{u_{i+1}^{P,n}+u_i^{P,n}}{2}
            + \frac{\Delta t}{\Delta x}\,
              \lambda_i[v]\,
              \bigl|u_{i+1}^{P,n}-u_i^{P,n}\bigr|.
\end{aligned}
\]
We notice that the solution, denoted $u^{P,n+1}$, of the scheme~\eqref{3sch:IC}--\eqref{3sch:eq} 
at step $n+1$  is given by
\begin{align*}
    u^{P,n+1}_i = F_i(u^{P,n+1}) \quad \forall i \in \I.
\end{align*}
We now aim to show that \( F = (F_i)_{i \in \I} : \mathcal{U}^\alpha \to \mathcal{U}^\alpha \) is a well-defined contraction mapping. For this purpose, let \( v = (v_i)_{i \in \I} \in \mathcal{U}^\alpha \). We observe that
\begin{align*}
    F_i(v) = \left( \frac{1}{2} + \dfrac{\Delta t}{\Delta x} \lambda_i[v] \, \text{sign}(u_{i+1}^{P, n} - u_i^{P, n}) \right) u_{i+1}^{P, n}
    + \left( \frac{1}{2} - \dfrac{\Delta t}{\Delta x} \lambda_i[v] \, \text{sign}(u_{i+1}^{P, n} - u_i^{P, n}) \right) u_i^{P, n}.
\end{align*}
To prove that $F_i(v)$ is a convex combination of \( u_i^{P, n} \) and \( u_{i+1}^{P, n} \), it remains to verify that both coefficients are non-negative and sum to 1. We have, by using \eqref{3Normcesaro}, that
\begin{align*}
   \left|\dfrac{\Delta t}{\Delta x}\lambda_i[v]\right| \leq \dfrac{\Delta t}{\Delta x}\left( 5 \alpha\|\K\|_{L^1(\R)} + \|a\|_{L^{\infty}(0, T)}  \right).
\end{align*}
Thanks to the CFL condition \eqref{3deltat/deltax}, we get $\left|\dfrac{\Delta t}{\Delta x} \lambda_i[v]\right| \le \frac{1}{2}$. Therefore, $F_i(v)$ is indeed a convex combination of $u_i^{P, n}$ and $u_{i+1}^{P, n}$ which are in $[-\alpha,\alpha]$, and hence $F(v) \in \mathcal{U}^\alpha$. Now, we equip \( \mathcal{U}^\alpha \) with the \( \ell^\infty \) norm:
\[
\|v\|_{\ell^{\infty}(I_P)} := \max_{i \in \I} |v_i|.
\]
For \( v, w \in \mathcal{U}^\alpha \), we compute:
\begin{align*}
    \|F(v) - F(w)\|_{\ell^{\infty}(I_P)} &= \max_{i \in \I} |F_i(v) - F_i(w)|\\
    &= \dfrac{\Delta t}{\Delta x} \max_{i \in \I}\left|u_{i + 1}^{P, n} - u_i^{P, n}\right| \left|\lambda_i[v] - \lambda_i[w]\right|\\
    &= \dfrac{\Delta t}{\Delta x} \max_{i \in \I}\left|u_{i + 1}^{P, n} - u_i^{P, n}\right| \left| \sum_{j \in \I} \Delta x \, \sigma_{M, j} \, (v_{i - j}-  w_{i - j})\right|\\
    &\le 10 \alpha \dfrac{\Delta t}{\Delta x} \|\K\|_{L^1(I_P)}\|v - w\|_{\ell^{\infty}(I_P)}.
\end{align*}
Thus, using \eqref{3deltat/deltax}, we deduce that $F$ is a contraction map on $\mathcal{U}^{\alpha}$. Therefore, by Banach fixed point theorem, we deduce that there exists a unique solution, denoted $u^{P,n+1}$ of \eqref{3sch:IC}-\eqref{3sch:eq} belonging to $\mathcal{U}^{\alpha}$. Moreover, we notice that
\begin{align*}
\theta^{P, n+1}_{i+1/2} = \left(\dfrac12 + \dfrac{\Delta t}{\Delta x} \, \lambda_{i+1}[u^{P, n + 1}]\, \mathrm{sign}\left(\theta^{P, n}_{i+3/2}\right)\right) \theta^{P, n}_{i+3/2} +\left(\dfrac12- \dfrac{\Delta t}{\Delta x} \, \lambda_i[u^{P, n + 1}] \,\mathrm{sign}\left(\theta^{P, n}_{i+1/2} \right) \right) \theta^{P, n}_{i+1/2}.
\end{align*}
Thanks to the CFL condition \eqref{3deltat/deltax}, we find that the coefficients are positive, and then
\begin{align*}
\left|\theta^{P, n+1}_{i+1/2}\right| \le \left(\dfrac12 + \dfrac{\Delta t}{\Delta x} \, \lambda_{i+1}[u^{P, n + 1}]\, \mathrm{sign}\left(\theta^{P, n}_{i+3/2}\right)\right) \left|\theta^{P, n}_{i+3/2}\right| +\left(\dfrac12- \dfrac{\Delta t}{\Delta x} \, \lambda_i[u^{P, n + 1}] \,\mathrm{sign}\left(\theta^{P, n}_{i+1/2} \right) \right) \left|\theta^{P, n}_{i+1/2}\right|.
\end{align*}
Now, by doing the summation over $i \in \I$, we get
\begin{multline*}
\sum_{i \in \I}\left|\theta^{P, n+1}_{i+1/2}\right| \le \sum_{i \in \I}\left(\dfrac12 + \dfrac{\Delta t}{\Delta x} \, \lambda_{i+1}[u^{P, n + 1}]\, \mathrm{sign}\left(\theta^{P, n}_{i+3/2}\right)\right) \left|\theta^{P, n}_{i+3/2}\right|\\ + \sum_{i \in \I}\left(\dfrac12- \dfrac{\Delta t}{\Delta x} \, \lambda_i[u^{P, n + 1}] \,\mathrm{sign}\left(\theta^{P, n}_{i+1/2} \right) \right) \left|\theta^{P, n}_{i+1/2}\right|.
\end{multline*}
Thanks to the periodicity that we have on $u^{P, n + 1}_i$ and $\theta^{P, n}_i$, we can then express the above inequality as
\begin{multline*}
\sum_{i \in \I}\left|\theta^{P, n+1}_{i+1/2}\right| \le \sum_{i \in \I}\left(\dfrac12 + \dfrac{\Delta t}{\Delta x} \, \lambda_{i}[u^{P, n + 1}]\, \mathrm{sign}\left(\theta^{P, n}_{i+1/2}\right)\right) \left|\theta^{P, n}_{i+1/2}\right|\\ + \sum_{i \in \I}\left(\dfrac12- \dfrac{\Delta t}{\Delta x} \, \lambda_i[u^{P, n + 1}] \,\mathrm{sign}\left(\theta^{P, n}_{i+1/2} \right) \right) \left|\theta^{P, n}_{i+1/2}\right|.
\end{multline*}
Therefore, we obtain
\begin{align*}
    \sum_{i \in \I}\left|\theta^{P, n+1}_{i+1/2}\right| \le \sum_{i \in \I}\left|\theta^{P, n}_{i+1/2}\right|,
\end{align*}
so that
\begin{align*}
    \operatorname{TV}(u^{P, n+1}) \leq \operatorname{TV}(u^{P, n}) \leq \operatorname{TV}(u^{P, 0}).
\end{align*}
It is important to note that, since $u^P_0 \in W^{1,1}(I_P)$, its total variation is equal to $\|\partial_x u^P_0\|_{L^1(I_P)}$.
Hence, we conclude by induction that the scheme~\eqref{3sch:IC}-\eqref{3sch:eq} admits a unique solution $u^{P,n}$ for any $n \in \{1,\ldots,N_T\}$ satisfying, thanks to Lemma~\ref{3main_2_lem.init.perio.to.nonperiod},
\[
\max_{i\in\I}\, |u^{P,n}_i| \leq \|u^P_0\|_{L^\infty(I_P)} \leq 2 \|u_0\|_{L^\infty(\R)},
\]
and
\begin{align*}
\operatorname{TV}\!\left(u^{P,n}\right) \leq \operatorname{TV}\!(u^{P, 0}) \le \|\partial_x u_0\|_{L^1(\R)}.
\end{align*}

\subsection{Discrete gradient entropy estimate}\label{3subsec:gradent}

In order to conclude the proof of Theorem~\ref{3theorem-properties}, it remains to establish the discrete gradient entropy estimate~\eqref{3gradient_entropy_estimate}. In particular, our analysis will rely on the three following lemmas.

\begin{lemma}[{\cite[Lemma 3.3]{MoMo14}}]
\label{3technical}
Let $\gamma_m > 1$ and $f$ the convex function given by~\eqref{3def:entropydensity}. Then, there exists a nonnegative function $g$ and a constant $C_{\gamma_m} > 0$ (only depending on $\gamma_m$) such that, for all $\theta > 0$ and $\gamma \in (0, \gamma_m)$, it holds
 \begin{equation*}%\label{3simple eq}
     f\bigg(\frac{\theta}{\gamma}\bigg) \geq \frac{1}{\gamma}f(\theta) - \frac{1}{\gamma}g(\theta,  \gamma)\ln(\gamma),
 \end{equation*}
 and $$|\theta - g(\theta, \gamma)| \leq C_{\gamma_m} = \frac{\gamma_m - 1}{e\ln(\gamma_m)}.$$
\end{lemma}
\begin{lemma}[{\cite[Lemma 3.4]{MoMo14}}] \label{3convexity inequality of f} 
Let $a_k$ and $\theta_k$ be two finite sequences of nonnegative real numbers such that $0< \sum_k a_k < 2$. Then, defining $\theta = \sum_k a_k \theta_k$, the following inequality holds
\[
f(\theta) \leq \sum_k a_k f(\theta_k) + g\left(\theta,\sum_k a_k \right) \, \ln\left(\sum_k a_k\right),
\]
where $f$ is given by~\eqref{3def:entropydensity} and $g$ is defined in Lemma \ref{3technical} for $\gamma_m = 2.$
\end{lemma}
\begin{lemma}[{\cite[Lemma 1]{ZHIZ25}}] \label{3fourier series of kpdelta negative} 
Let assumption~\textbf{\emph{(H2)}} holds. Then, the Fourier coefficients of the $2P$-periodic function $\K^P_M$ and $\sigma_M^P$ are nonpositive real numbers.
\end{lemma}
Now, we are ready to prove estimate ~\eqref{3gradient_entropy_estimate}. For this purpose, we first notice that
\begin{align}\label{3eq:theta}
\theta^{P, n+1}_{i+1/2} = \dfrac{\theta^{P, n}_{i+3/2} +\theta^{P, n}_{i+1/2} }{2} + \dfrac{\Delta t}{\Delta x} \, \lambda_{i+1}[u^{P, n + 1}] \, \left|\theta^{P, n}_{i+3/2}\right|- \dfrac{\Delta t}{\Delta x} \, \lambda_i[u^{P, n + 1}]\, \left|\theta^{P, n}_{i+1/2} \right|,
\end{align}
or equivalently
\begin{multline*}
\theta^{P, n+1}_{i+1/2} = \left(\dfrac12 + \dfrac{\Delta t}{\Delta x} \, \lambda_{i+1}[u^{P, n + 1}]\, \mathrm{sign}\left(\theta^{P, n}_{i+3/2} \right)\right)\theta^{P, n}_{i+3/2}+\left(\dfrac12- \dfrac{\Delta t}{\Delta x} \, \lambda_i[u^{P, n + 1}] \,\mathrm{sign}\left(\theta^{P, n}_{i+1/2} \right) \right)\theta^{P, n}_{i+1/2}.
\end{multline*}
Thus, under the CFL condition~\eqref{3deltat/deltax}, the coefficients in the right hand side of the previous equality are nonnegative, and we get
\begin{align}\label{3absolute theta}
\left|\theta^{P, n+1}_{i+1/2}\right| = \dfrac{\left|\theta^{P, n}_{i+3/2}\right|+\left|\theta^{P, n}_{i+1/2}\right|}{2} + \dfrac{\Delta t}{\Delta x} \, \lambda_{i+1}[u^{P, n + 1}] \, \theta^{P, n}_{i+3/2} - \dfrac{\Delta t}{\Delta x} \, \lambda_i[u^{P, n + 1}] \, \theta^{P, n}_{i+1/2}.
\end{align}
Let us define the positive and negative parts of $\theta^{P, n+1}_{i+1/2}$, for all $0 \le n \le N_T$ and $i \in \I$, by 
\begin{align}\label{3def_of_theta_sign}
\theta^{+,P,n}_{i+1/2}
= \frac{\left|\theta^{P,n}_{i+1/2}\right| + \theta^{P,n}_{i+1/2}}{2},
\qquad
\theta^{-,P,n}_{i+1/2}
= \frac{\left|\theta^{P,n}_{i+1/2}\right| - \theta^{P,n}_{i+1/2}}{2}.
\end{align}
Hence, for $\theta^{+, P, n+1}_{i+1/2}$ and $\theta^{-,  P, n+1}_{i+1/2}$ defined in \eqref{3def_of_theta_sign}, we obtain, by adding \eqref{3eq:theta} and \eqref{3absolute theta},
\begin{align}\label{3ineq:theta+}
\theta^{+, P, n+1}_{i+1/2} \leq \widetilde{\theta}^{+, P, n+1}_{i+1/2} := \dfrac{\theta^{+,P, n}_{i+3/2}+\theta^{+,P,  n}_{i+1/2}}{2} + \dfrac{\Delta t}{\Delta x} \, \lambda_{i+1}[u^{P, n + 1}] \, \theta^{+, P, n}_{i+3/2} - \dfrac{\Delta t}{\Delta x} \, \lambda_i[u^{P, n + 1}] \, \theta^{+,P, n}_{i+1/2},
\end{align}
and, by subtracting \eqref{3eq:theta} from \eqref{3absolute theta},
\begin{align}\label{3ineq:theta-}
\theta^{-,  P, n+1}_{i+1/2} \leq \widetilde{\theta}^{-, P, n+1}_{i+1/2} := \dfrac{\theta^{-, P, n}_{i+3/2}+\theta^{-, P, n}_{i+1/2}}{2} - \dfrac{\Delta t}{\Delta x} \, \lambda_{i+1}[u^{P, n + 1}] \, \theta^{-, P, n}_{i+3/2} + \dfrac{\Delta t}{\Delta x} \, \lambda_i[u^{P, n + 1}] \, \theta^{-, P, n}_{i+1/2}.
\end{align}
Let us first work with~\eqref{3ineq:theta+}. To do so, we introduce the constants
\begin{align*}
a_1^+ := \dfrac12 +  \dfrac{\Delta t}{\Delta x} \, \lambda_{i+1}[u^{P, n + 1}], \quad a_2^+ := \dfrac12 - \dfrac{\Delta t}{\Delta x} \, \lambda_i[u^{P, n + 1}],
\end{align*}
so that
\begin{align*}
\widetilde{\theta}^{+, P, n+1}_{i+1/2} := a_1^+ \, \theta^{+, P,  n}_{i+3/2} + a_2^+ \, \theta^{+, P, n}_{i+1/2}.
\end{align*}
Moreover, we introduce $\mu^{+,n+1}_{i+1/2} = 1 - (a^+_1 + a^+_2)$. Then, thanks to ~\eqref{3deltat/deltax}, we have
\begin{align*}
1-\mu^{+,n+1}_{i+1/2} = a_1^+ + a_2^+ = 1 - \dfrac{\Delta t}{\Delta x} \left(\lambda_i[u^{P, n + 1}]-\lambda_{i+1}[u^{P, n + 1}] \right) \in (0,2).
\end{align*}
Therefore, according to Lemma~\ref{3convexity inequality of f}, we have
\begin{multline*}
f\left(\theta^{+,P, n+1}_{i+1/2}\right) \leq \dfrac{f\left(\theta^{+, P, n}_{i+3/2}\right)+f\left(\theta^{+, P, n}_{i+1/2}\right)}{2} + \dfrac{\Delta t}{\Delta x} \, \lambda_{i+1} [u^{P, n+1}]f\left(\theta^{+, P, n}_{i+3/2}\right) - \dfrac{\Delta t}{\Delta x} \, \lambda_i[u^{P, n+1}] \, f\left(\theta^{+, P, n}_{i+1/2}\right) \\
+ g\left(\widetilde{\theta}^{+, P, n+1}_{i+1/2},1-\mu^{+,n+1}_{i+1/2}\right) \, \ln\left(1-\mu^{+,n+1}_{i+1/2}\right). 
\end{multline*}
Now, we multiply the above inequality by $\Delta x$ and sum over all $i \in \I$. By reordering the terms and using the inequality $\ln(1-\mu) \le -\mu$ for all $\mu < 1$, we obtain
\begin{align*}
\sum_{i\in\I} \Delta x \, f\left(\theta^{+, P, n+1}_{i+1/2}\right) \leq \sum_{i \in \I} \Delta x  \, f\left(\theta^{+,P, n}_{i+1/2}\right)-\sum_{i \in \I} \Delta x\, g\left(\widetilde{\theta}^{+, P, n+1}_{i+1/2},1-\mu^{+,n+1}_{i+1/2}\right) \, \mu^{+,n+1}_{i+1/2}. 
\end{align*}
Moreover, we rewrite the last term in the right hand side of the previous inequality as 
\begin{multline*}
    -\sum_{i \in \I} \Delta x\, g\left(\widetilde{\theta}^{+, P, n+1}_{i+1/2},1-\mu^{+,n+1}_{i+1/2}\right) \, \mu^{+,n+1}_{i+1/2} \\=  -\sum_{i \in \I} \Delta x\, \left(g\left(\widetilde{\theta}^{+, P, n+1}_{i+1/2},1-\mu^{+,n+1}_{i+1/2}\right) - \widetilde{\theta}^{+, P, n+1}_{i+1/2}\right)\, \mu^{+,n+1}_{i+1/2} -\sum_{i \in \I} \Delta x\, \left(\widetilde{\theta}^{+,P, n+1}_{i+1/2}\right)\, \mu^{+,n+1}_{i+1/2}.
\end{multline*}
For the first term in the right hand side, thanks to Lemma \ref{3technical}, with $\gamma_m = 2$, we have 
\begin{align*}
   \Bigg|-\sum_{i \in \I} \Delta x\, \left(g\left(\widetilde{\theta}^{+, P, n+1}_{i+1/2},1-\mu^{+,n+1}_{i+1/2}\right) - \widetilde{\theta}^{+, P, n+1}_{i+1/2}\right)&\, \mu^{+,n+1}_{i+1/2}\Bigg| \le \frac{1}{e \ln(2)} \sum_{i \in \I} \Delta x \left|\mu^{+,n+1}_{i+1/2} \right| \\&= \frac{1}{e \ln(2)} \sum_{i \in \I} \Delta x \left|\dfrac{\Delta t}{\Delta x} \left(\lambda_i[u^{P, n + 1}]-\lambda_{i+1}[u^{P, n + 1}] \right)\right|\\
    &\le \dfrac{5 \Delta t \|\K\|_{L^1(\R)}}{e \ln(2)} \|\partial_x u_0\|_{L^1(\R)}, 
\end{align*}
where for the last inequality we have used~\eqref{3TV:ineq} and lemma~\ref{3lemma tech kernel} in Appendix~\ref{3app.techni}. Hence, we end up with
\begin{align}\label{3ineq:ent1}
\sum_{i\in\I} \Delta x \, f\left(\theta^{+, P, n+1}_{i+1/2}\right) \leq \sum_{i \in \I} \Delta x  \, f\left(\theta^{+,P, n}_{i+1/2}\right) + C \,\Delta t -\sum_{i \in \I} \Delta x\, \widetilde{\theta}^{+, P, n+1}_{i+1/2} \, \mu^{+,n+1}_{i+1/2},
\end{align}
where $C>0$ is a constant which depends on $\|\K\|_{L^1(\R)}$ and $\|\partial_x u_0\|_{L^1(\R)}$. Similarly, we rewrite the term $\widetilde{\theta}^{-, P, n+1}_{i+1/2}$ as
\begin{align*}
\widetilde{\theta}^{-, P, n+1}_{i+1/2} = a_1^- \, \theta^{+, P, n}_{i+3/2} + a_2^- \, \theta^{+, P, n}_{i+1/2},
\end{align*}
with
\begin{align*}
a_1^- := \dfrac12 -  \dfrac{\Delta t}{\Delta x} \, \lambda_{i+1}[u^{P, n + 1}], \quad a_2^- := \dfrac12 + \dfrac{\Delta t}{\Delta x} \, \lambda_i[u^{P, n + 1}],
\end{align*}
and we introduce the term $\mu^{-,n+1}_{i+1/2} = 1 - (a^-_1 + a^-_2)$, so that, under the CFL condition~\eqref{3deltat/deltax}, we get
\begin{align*}
1-\mu^{-,n+1}_{i+1/2} = a_1^- + a_2^- = 1 + \dfrac{\Delta t}{\Delta x} \left(\lambda_i[u^{P, n + 1}]-\lambda_{i+1}[u^{P, n+ 1}] \right) \in (0,2).
\end{align*}
Therefore, arguing as before, we get
\begin{align}
\sum_{i\in\I} \Delta x \, f\left(\theta^{-, P,n+1}_{i+1/2}\right) &\leq \sum_{i \in \I} \Delta x  \, f\left(\theta^{-, P, n}_{i+1/2}\right) + C\,\Delta t -\sum_{i \in \I} \Delta x\, \widetilde{\theta}^{-, P,n+1}_{i+1/2} \, \mu^{-,n+1}_{i+1/2}\nonumber\\
&\leq \sum_{i \in \I} \Delta x  \, f\left(\theta^{-, P, n}_{i+1/2}\right) + C\,\Delta t +\sum_{i \in \I} \Delta x\, \widetilde{\theta}^{-, P,n+1}_{i+1/2} \, \mu^{+,n+1}_{i+1/2},\label{3ineq:ent2}
\end{align}
where we have used the relation $\mu^{-,n+1}_{i+1/2} = - \mu^{+,n+1}_{i+1/2}$. Now, we sum estimates~\eqref{3ineq:ent1} and~\eqref{3ineq:ent2} and we obtain
\begin{multline*}
\sum_{i \in \I} \Delta x \left[f\left(\theta^{+, P, n+1}_{i+1/2}\right) + f\left(\theta^{-, P, n+1}_{i+1/2}\right) \right] \leq \sum_{i \in \I} \Delta x \left[f\left(\theta^{+, P, n}_{i+1/2}\right) + f\left(\theta^{-, P, n}_{i+1/2}\right) \right] \\
+ C\,\Delta t + \sum_{i \in \I} \Delta x\,\left( \widetilde{\theta}^{-, P, n+1}_{i+1/2}-\widetilde{\theta}^{+, P, n+1}_{i+1/2} \right) \, \mu^{+,n+1}_{i+1/2}.
\end{multline*}
It remains to notice that for every $i\in\I$ we have
\begin{align*}
\widetilde{\theta}^{-, P, n+1}_{i+1/2}-\widetilde{\theta}^{+, P, n+1}_{i+1/2} &= \dfrac{\theta^{-, P, n}_{i+3/2}+\theta^{-, P, n}_{i+1/2}}{2} - \dfrac{\Delta t}{\Delta x} \, \lambda_{i+1}[u^{P, n + 1}] \, \theta^{-, P, n}_{i+3/2} + \dfrac{\Delta t}{\Delta x} \, \lambda_i[u^{P, n + 1}] \, \theta^{-, P,n}_{i+1/2}\\
&- \left[\dfrac{\theta^{+, P, n}_{i+3/2}+\theta^{+, P, n}_{i+1/2}}{2} + \dfrac{\Delta t}{\Delta x} \, \lambda_{i+1}[u^{P, n + 1}] \, \theta^{+, P, n}_{i+3/2} - \dfrac{\Delta t}{\Delta x} \, \lambda_i[u^{P, n + 1}] \, \theta^{+, P, n}_{i+1/2} \right]\\
&=-\left[ \dfrac{\theta^{P, n}_{i+3/2}+\theta^{P, n}_{i+1/2}}{2} + \dfrac{\Delta t}{\Delta x} \, \lambda_{i+1}[u^{P, n + 1}] \, \left|\theta^{P, n}_{i+3/2}\right| - \dfrac{\Delta t}{\Delta x} \, \lambda_i[u^{P, n + 1}] \, \left|\theta^{P, n}_{i+1/2} \right| \right],
\end{align*}
so that
\begin{align*}
\widetilde{\theta}^{-, P, n+1}_{i+1/2}-\widetilde{\theta}^{+, P, n+1}_{i+1/2} = -\theta^{P, n+1}_{i+1},
\end{align*}
where we have used the relations
\begin{align*}
\theta^{P, n}_{i+1/2} = \theta^{+, P, n}_{i+1/2} - \theta^{-, P, n}_{i+1/2}, \quad \left|\theta^{P, n}_{i+1/2}\right| = \theta^{+, P, n}_{i+1/2} + \theta^{-, P, n}_{i+1/2} \quad \forall i \in \I,
\end{align*}
and the equality~\eqref{3eq:theta}. We conclude that it holds
\begin{multline*}
\sum_{i \in \I} \Delta x \left[f\left(\theta^{+, P, n+1}_{i+1/2}\right) + f\left(\theta^{-, P, n+1}_{i+1/2}\right) \right] \leq \sum_{i \in \I} \Delta x \left[f\left(\theta^{+,P, n}_{i+1/2}\right) + f\left(\theta^{-, P, n}_{i+1/2}\right) \right] \\
+ C\,\Delta t - \sum_{i \in \I} \Delta x \, \theta^{P, n+1}_{i+1/2} \, \mu^{+,n+1}_{i+1/2}.
\end{multline*}
Finally, applying the definition of $\mu^{+,n+1}_{i+1/2}$ and~\eqref{3def:lambda} we end up with
\begin{multline}\label{3ineg tech ent}
\sum_{i \in \I} \Delta x \left[f\left(\theta^{+, P, n+1}_{i+1/2}\right) + f\left(\theta^{-, P, n+1}_{i+1/2}\right) \right] \leq \sum_{i \in \I} \Delta x \left[f\left(\theta^{+, P, n}_{i+1/2}\right) + f\left(\theta^{-, P, n}_{i+1/2}\right) \right] \\
+ C\,\Delta t +\Delta t \sum_{i \in \I} \Delta x \, \theta^{P, n+1}_{i+1/2} \sum_{j\in\I} \Delta x \, \sigma_{M, j}^P \, \theta^{P, n+1}_{i-j+1/2}.
\end{multline}
We need to use the Discrete Fourier Transform (DFT). We first define the vector
\[
\theta^{P,n+1} = \left(\theta^{P,n+1}_{i+1/2}\right)_{i\in\I},
\]
and, for any vector $v=(v_i)_{i\in\I}$, we introduce the discrete counterpart of the periodic convolution product~\eqref{3conv.prod.perio}, still denoted $\star$, between $v$ and $\theta^{P,n+1}$ as
\begin{align*}
    \left(v \star \theta^{P,n+1}\right)_i := \sum_{j \in \I} \Delta x \, v_j \, \theta^{P,n+1}_{i-j+1/2} = \sum_{j \in \I}  v_j \, \left(u^{P,n+1}_{i+1-j}-u^{P,n+1}_{i-j}\right), \quad \forall i \in \I.
\end{align*}
Then, defining $ \bar{\sigma}_M^P= (\sigma_M^P(x_i))_{i\in\I}$, we notice that we can rewrite the last term in the right hand side of~\eqref{3ineg tech ent} as
\begin{align*}
\Delta t \sum_{i \in \I} \Delta x \, \theta^{P, n+1}_{i+1/2} \sum_{j\in\I} \Delta x \, \sigma_{M, j}^P \, \theta^{P, n+1}_{i-j+1/2} = \Delta t \sum_{i\in\I} \Delta x \, \theta^{P,n+1}_{i+1/2} \, \left( \bar{\sigma}_M^P \star \theta^{P,n+1}\right)_i. 
\end{align*}
In the sequel, for any $v = (v_i)_{i\in \I}$, we will denote by $c^d(v) = (c^d_j(v))_{j\in\I}$ the vector of the coefficients of the DFT of $v$ where
\begin{align*}
    c^d_j(v) := \frac{1}{2N} \sum_{\ell \in \I} v_\ell \, e^{-i\pi \, j \, \ell/N}, \quad \forall j \in \I.
\end{align*}
Therefore, applying the discrete version of Parseval's equality as well as the property of the DFT with respect to the discrete convolution product, we get
\begin{align*}
\Delta t \sum_{i \in \I} \Delta x \, \theta^{P, n+1}_{i+1/2} \sum_{j\in\I} \Delta x \, \sigma_{M, j}^P \, \theta^{P, n+1}_{i-j+1/2} = (2N)^2 \, \Delta t  \, \sum_{j \in \I} \Delta x \, c^d_j\left(\bar{\sigma}_M^P\right) \, \left|c^d_j\left(\theta^{P,n+1}\right)\right|^2.
\end{align*}
Let us now show that the coefficients of the DFT of the vector $\bar{\sigma}_M^P$ are real nonpositive numbers. For this purpose, let $j\in\I$ be fixed. Then, by definition, we have
\begin{align*}
    c^d_j\left(\bar{\sigma}_M^P\right) = \frac{1}{2N} \sum_{\ell \in \I} \sigma_M^P(x_\ell) \, e^{-i\pi \, j \, \ell/N} &= \frac{1}{2NM} \sum_{\ell\in\I} \sum_{k=0}^{M-1} \sum_{|m|\leq k} c_m(\K^P_M) \, e^{i\pi (m-j) \ell/N}\\
    &=\frac{1}{2NM} \sum_{k=0}^{M-1} \sum_{|m|\leq k} c_m(\K^P_M) \, \sum_{\ell\in\I} e^{i\pi (m-j) \ell/N}.
\end{align*}
Hence, since the last sum is either equal to $0$ or $2N$, depending on the values of $m$, we conclude that $c^d_j(\bar{\sigma}_M^P)$ is a sum of the Fourier coefficients $c_m(\K^P_M)$ for $|m|<M$. Therefore, thanks to Lemma~\ref{3fourier series of kpdelta negative}, we obtain
\begin{align}\label{3J3}
 \Delta t \sum_{i \in \I} \Delta x \, \theta^{P, n+1}_{i+1/2} \sum_{j\in\I} \Delta x \, \sigma_{M, j}^P \, \theta^{P, n+1}_{i-j+1/2}  \leq 0.
\end{align}
Hence, 
\begin{align*}
\sum_{i \in \I} \Delta x \left[f\left(\theta^{+, P, n+1}_{i+1/2}\right) + f\left(\theta^{-, P, n+1}_{i+1/2}\right) \right] \leq \sum_{i \in \I} \Delta x \left[f\left(\theta^{+, P, n}_{i+1/2}\right) + f\left(\theta^{-, P, n}_{i+1/2}\right) \right]
+ C\,\Delta t.
\end{align*}
Therefore, for any $0\leq n \leq N_T-1$, we obtain
\begin{align}\label{3ineq:disent}
\sum_{i \in \I} \Delta x \left[f\left(\theta^{+, P, n+1}_{i+1/2}\right) + f\left(\theta^{-, P, n+1}_{i+1/2}\right) \right] \leq \sum_{i \in \I} \Delta x \left[f\left(\theta^{+, P, 0}_{i+1/2}\right) + f\left(\theta^{-, P, 0}_{i+1/2}\right) \right] + C\,T.
\end{align}
Now, let $\theta \in \R$ be fixed. Then, we have $|\theta| = \theta^+ + \theta^-$ with
\begin{align*}
\theta^+ = \dfrac{|\theta|+\theta}{2}, \quad \theta^- = \dfrac{|\theta|-\theta}{2}.
\end{align*}
Adapting slightly the proofs of Lemma~\ref{3technical} and Lemma~\ref{3convexity inequality of f} (see also~\cite{MoMo14}), with $\gamma=2$ and for instance $\gamma_m = 3$, we obtain
\begin{align*}
f(|\theta|) &\leq f(\theta^+ + \theta^-) \leq f(\theta^+) + f(\theta^-) + g(|\theta|,2) \, \ln(2)\\
&\leq f(\theta^+) + f(\theta^-) + \left|g(|\theta|,2)-|\theta|\right| \, \ln(2) + |\theta| \, \ln(2)\\
&\leq f(\theta^+) + f(\theta^-) + \dfrac{2 \, \ln(2)}{e \, \ln(3)} + |\theta| \, \ln(2).
\end{align*}
Applying this inequality to $|\theta^{P, n+1}_{i+1/2}| = \theta^{+, P, n+1}_{i+1/2}+\theta^{-, P, n+1}_{i+1/2}$, we get thanks to~\eqref{3ineq:disent}, for any $0 \leq n \leq N_T-1$, the following estimate
\begin{align*}
\sum_{i \in \I} \Delta x \, f\left(\left|\theta^{P, n+1}_{i+1/2}\right|\right)  \leq \ln(2)\sum_{i \in \I} \Delta x \left|\theta^{P, 0}_{i+1/2}\right| + 2 e^{-1}+ \sum_{i \in \I} \Delta x \left[f\left(\theta^{+, P, 0}_{i+1/2}\right) + f\left(\theta^{-, P, 0}_{i+1/2}\right) \right]+ C\, T ,
\end{align*}
thus, we obtain 
\begin{align*}
\sum_{i \in \I} \Delta x \, f\left(\left|\theta^{P, n+1}_{i+1/2}\right|\right)  \leq \ln(2)\,\|\partial_x u_0\|_{L^1(\R)} + 2 e^{-1} + 2 \sum_{i \in \I} \Delta f\left(\left|\theta^{P, 0}_{i+1/2}\right|\right) + C\, T,
\end{align*}
which is precisely (\ref{3gradient_entropy_estimate}).

\section{Proofs of the convergence results}\label{3Convergence of the scheme}

In this section, we prove Theorems~\ref{3theorem-convergence} and \ref{3theorem-limit-PM}. We begin with the proof of Theorem~\ref{3theorem-convergence}. For this purpose, we split our proof in several steps. First, in Section \ref{3proof of theorem 2-i} we establish some uniform in $\eps=(\Delta x, \Delta t)$ estimates, whereas in Section \ref{3discontinuous viscosity solution} we show the existence of a function $u^P_M$ such that, up to a subsequence, the sequence $(u^{P,\varepsilon_m}_M)$ converges toward $u^P_M$ (in the sense specified in the statement of Theorem \ref{3theorem-convergence}). Then, we identify in Section \ref{3viscosity solution} this function as a solution of~\eqref{31.equPaux} in the viscosity sense. Finally, in Section~\ref{3existence of viscosity solution}, we prove the convergence toward the original model~\ref{3eq:eikonal}. Our convergence proof is similar to the one of \cite{ZHIZ25, alzohbi_elhajj_jazar_2022} and relies on Lemma \ref{3Llog L estimate} stated in Appendix \ref{3app.techni}.

\subsection{Uniform estimates} \label{3proof of theorem 2-i} We notice that
\begin{multline*}
    \|u^{P, \varepsilon}_M\|_{L^{\infty}((0, T) \times I_P)} \le \left( \frac{t - t_n}{\Delta t} \right)
\left[
\left( \frac{x - x_i}{\Delta x} \right) \|u^{P}_0\|_{L^{\infty}(I_P)}
+ \left( 1 - \frac{x - x_i}{\Delta x} \right) \|u^{P}_0\|_{L^{\infty}(I_P)}
\right]  \\
\quad +
\left( 1 - \frac{t - t_n}{\Delta t} \right)
\left[
\left( \frac{x - x_i}{\Delta x} \right) \|u^{P}_0\|_{L^{\infty}(I_P)}
+ \left( 1 - \frac{x - x_i}{\Delta x} \right) \|u^{P}_0\|_{L^{\infty}(I_P)}
\right]. 
\end{multline*}
Therefore, thanks to Theorem~\ref{3theorem-properties}, we get
$$\|u^{P, \varepsilon}_M\|_{L^{\infty}((0, T) \times I_P)} \le \|u^{P}_0\|_{L^{\infty}(I_P)}  \le 2\|u_0\|_{L^{\infty}(\R)}.$$
Now we want to prove inequality \eqref{3bounded ux}. Thanks to the definition~\eqref{3defQ1ext} of $u^{P,\eps}_M$, we observe that for any $(x,t) \in (x_i,x_{i+1}) \times [t_n , t_{n+1}]$ it holds
\begin{equation}\label{3uepsilonx}
\partial_x u^{P, \eps}_{M} (x,t) = \bigg(\frac{t - t_n}{\Delta t}\bigg) \,\theta_{i +1/2}^{P, n + 1} + \bigg(1 - \frac{t - t_n}{\Delta t}\bigg) \, \theta_{i +1/2}^{P, n }.
\end{equation}
Therefore, we obtain for a given $t \in [t_n, t_{n + 1}]$ (for some $0\leq n \leq N_T-1)$
\begin{align*}
\int_{I_P} \left|\partial_x u^{P, \eps}_{M}(x,t)\right|\, \dd x &= \sum_{i \in \I} \int_{x_i}^{x_{i + 1}} \left|\partial_x u^{P,\eps}_M(x,t)\right|\, \dd x\\ 
&\leq \sum_{i\in \I}  \Delta x\, \bigg(\frac{t - t_n}{\Delta t}\bigg)\, \left|\theta_{i +1/2}^{P, n + 1}\right| + \sum_{i\in \I} \Delta x \, \bigg(1 - \frac{t - t_n}{\Delta t}\bigg) \, \left|\theta_{i +1/2}^{P, n }\right|,
\end{align*}
so that
\begin{align*}
     \int_{I_P} \left|\partial_x u^{P, \eps}_{M}(x,t)\right|\, \dd x &\leq \bigg(\frac{t - t_n}{\Delta t}\bigg)\,\sum_{i\in \I} \Delta x \, \left|\theta_{i +1/2}^{P, n + 1}\right| + \bigg(1 - \frac{t - t_n}{\Delta t}\bigg)\, \sum_{i\in \I} \Delta x \, \left|\theta_{i +1/2}^{P, n }\right| \\
    & = \bigg(\frac{t - t_n}{\Delta t}\bigg)\, \sum_{i\in\I} \left(\left|u^{P,n+1}_{i+1}-u^{P,n+1}_i\right|\right)+\bigg(1 - \frac{t - t_n}{\Delta t}\bigg)\, \sum_{i\in\I} \left(\left|u^{P,n}_{i+1}-u^{P,n}_i\right|\right).
\end{align*}
 Applying estimate~\eqref{3TV:ineq} in Theorem~\ref{3theorem-properties}, we end up with
\begin{equation*}
\|\partial_x u^{P, \eps}_M\|_{L^{\infty}(0, T; L^1(I_P))} \leq \, \|\partial_x u_0\|_{L^1(\R)}.
\end{equation*}

It remains to prove the inequality \eqref{3time estimate of u}. For this, using definition~\eqref{3defQ1ext}, we get, for any $(x,t) \in [x_i,x_{i+1}] \times (t_n,t_{n+1})$,
\begin{align*}
\partial_t u^{P, \eps}_M(x,t) &= \frac{1}{\Delta t}\left( \frac{x - x_i}{\Delta x} \right) 
\left[
u^{P, n+1}_{i+1}
- u^{P, n+1}_i
\right]- \frac{1}{\Delta t} \left( \frac{x - x_i}{\Delta x} \right) 
\left[
u^{P, n}_{i+1}
-  u^{P, n}_i
\right] + \frac{1}{\Delta t} \left[u^{P, n+1}_i  - u^{P, n}_i \right].
\end{align*}
Inserting \eqref{3sch:eq} yields
\begin{align*}
\partial_t u^{P, \eps}_M(x,t) &= \frac{1}{\Delta t}\left( \frac{x - x_i}{\Delta x} \right) 
\left[
u^{P, n+1}_{i+1}
- u^{P, n+1}_i
\right]- \frac{1}{\Delta t} \left( \frac{x - x_i}{\Delta x} \right) 
\left[
u^{P, n}_{i+1}
-  u^{P, n}_i
\right] \\
&+ \dfrac{1}{\Delta x}\lambda_i[u^{P, n + 1}] \, \left|u^{P, n}_{i+1}-u^{P, n}_i \right|+ \frac{1}{2\Delta t} \left[u^{P, n}_{i + 1}  - u^{P, n}_i \right].
\end{align*}
Moreover, since $x - x_i \leq \Delta x$ for any $i\in \I$, we obtain
\begin{align*}
\left|\partial_t u^{P, \eps}_M(x,t)\right| \le \frac{\Delta x}{\Delta t}
\left|
\theta_{i + 1/2}^{P, n + 1}
\right| + \frac32 \,\frac{\Delta x}{\Delta t} 
\left|
\theta_{i + 1/2}^{P, n}
\right| 
+ \left|\lambda_i[u^{P, n + 1}]\right| \, \left|
\theta_{i + 1/2}^{P, n + 1}
\right|.
\end{align*}
Therefore, thanks to Theorem~\ref{3theorem-properties} and estimate~\eqref{3Normcesaro}, we have
\begin{multline*}
\int_{I_P}\left|\partial_t u^{P, \eps}_M(x,t)\right| \, \dd x
\le  \frac{\Delta x}{\Delta t} \left(\sum_{i \in \I} 
\left|
\theta_{i + 1/2}^{P, n + 1}
\right|\Delta x + \frac32 \sum_{i \in \I} 
\left|
\theta_{i + 1/2}^{P, n}
\right| \Delta x \right)
\\ + \left(5 \|u^P_0\|_{L^{\infty}(\R)}\|\K\|_{L^1(\R)} + \|a\|_{L^{\infty}(0, T)}\right)\sum_{i \in \I} \, \left|
\theta_{i + 1/2}^{P, n + 1}
\right| \Delta x.
\end{multline*}
Then, applying the CFL condition~\eqref{3deltat/deltax} and estimate~\eqref{3TV:ineq}, we conclude that it holds
\begin{align}\label{3ineg conv partialtu}
\int_{I_P}\left|\partial_t u^{P, \eps}_M(x,t)\right| \, \dd x 
\le  6 \left[ 5 \|u^P_0\|_{L^{\infty}(\R)}\|\K\|_{L^1(\R)} + \|a\|_{L^{\infty}(0, T)}\right]\, \|\partial_x u_0\|_{L^1(\R)}.
\end{align}

Let us now prove that $\pa_x u^{P,\eps}_M$ is uniformly bounded in $L^\infty(0,T;L \log L(I_P))$. For this purpose, the main idea is to apply Lemma \ref{3Llog L estimate}. However, let us notice that the first term in the right hand side of \eqref{3gradient_entropy_estimate} depends on $\Delta x$ through the term $f\left(\left|\theta^{P,0}_{i+1/2}\right|\right)$. Therefore, we need to establish a uniform~w.r.t. $\eps$ estimate on this term.

\begin{lemma}\label{3lem.boundinitent}
Let the assumptions of Theorem \ref{3theorem-properties} hold. Then, there exists a constant $C_0>0$ 
only depending on $\|u_0\|_{L^{\infty}(\R)}, \, \|\partial_x u_0\|_{L^1(\R)}$ and $ \|\partial_x u_0\|_{L \log L(\R)}$,
 such that
\begin{align*}
    I_0=\sum_{i\in\I} \Delta x \, f\left(\left|\theta^{P,0}_{i+1/2}\right|\right) \leq C_0.
\end{align*}
\end{lemma}
\begin{proof} We start from the definition 
\[
\theta_{i+\frac{1}{2}}^{P,0}
= \frac{1}{\Delta x} \int_{x_i}^{x_{i+1}} \partial_x u_0^P(y)\,\mathrm{d}y.
\]
Then, by the triangle inequality,
\[
\big|\theta_{i+\frac{1}{2}}^{P,0}\big|
= \left|\frac{1}{\Delta x} \int_{x_i}^{x_{i+1}} \partial_x u_0^P(y)\,\mathrm{d}y\right|
\le \frac{1}{\Delta x} \int_{x_i}^{x_{i+1}} \big|\partial_x u_0^P(y)\big|\,\mathrm{d}y.
\]
Since the function \(f\) is convex and nondecreasing on \([0,\infty)\), we can apply monotonicity
followed by Jensen’s inequality to obtain
\[
f(|\theta_{i+\frac{1}{2}}^{P,0}|)
\le f\left(\frac{1}{\Delta x} \int_{x_i}^{x_{i+1}} |\partial_x u_0^P(y)|\,\mathrm{d}y\right)
\le \frac{1}{\Delta x} \int_{x_i}^{x_{i+1}} f\big(|\partial_x u_0^P(y)|\big)\,\mathrm{d}y.
\]
Hence,
\[
 f\left(\left|\theta^{P,0}_{i+1/2}\right|\right) 
\le \frac{1}{\Delta x} \int_{x_i}^{x_{i+1}} f\big(|\partial_x u_0^P(y)|\big)\,\mathrm{d}y.
\]
Applying \eqref{3f in llogl} from Lemma ~\ref{3Llog L estimate} yields
\begin{align*}
    I_0 = \sum_{i\in\I} \Delta x \, f\left(\left|\theta^{P,0}_{i+1/2}\right|\right) &\le \int_{I_P} f\big(|\partial_x u_0^P(y)|\big)\,\mathrm{d}y.\\
    &\le 1 + \|\partial_x u^P_0\|_{L\log L(I_P)} + \|\partial_xu_0^P\|_{L^1(I_P)}\ln(1 + \|\partial_x u^P_0\|_{L \log L(I_P)})
\end{align*}
Thanks to \eqref{3estimationLlogLP}, we show that $I_0$ is uniformly bounded with respect to $P$. Finally, by assumption \textbf{(H1)}, we obtain the desired result.
\end{proof}
We are now in position to prove the $L \log L$-boundedness presented in Theorem~\ref{3theorem-convergence}. First, for $(x,t)\in(x_i,x_{i+1})\times[t_n,t_{n+1}]$ with $i\in \I$ and $0\le n \le N_T-1$, we apply the triangular inequality on \eqref{3uepsilonx} and we get
\begin{equation*}
\left|\partial_x u^{P, \eps}_M(x,t) \right| \le  \bigg(\frac{t - t_n}{\Delta t}\bigg) \, \left|\theta_{i +1/2}^{P, n + 1}\right|  + \bigg(1 - \frac{t - t_n}{\Delta t}\bigg) \,  \left|\theta_{i +1/2}^{P, n }\right|. 
\end{equation*}
Now, using the convexity and the increasing property of $f$, we obtain
\[
f\left(\left|\partial_x u^{P, \eps}_M(x,t)\right|\right) \leq \bigg(\frac{t - t_n}{\Delta t}\bigg) \, f\left(\left|\theta_{i +1/2} ^{P, n + 1} \right|\right)+ \bigg(1 - \frac{t - t_n}{\Delta t}\bigg) \, f\left(\left|\theta_{i +1/2}^{P, n } \right|\right).
\]
Therefore, by integrating in space, we have
\begin{multline*}
\int_{I_P} f\left(\left|\partial_x u^{P, \eps}_M(x,t)\right|\right)\, \dd x \leq \bigg(\frac{t - t_n}{\Delta t}\bigg) \, \sum_{i\in\I} \Delta x\, f\left(\left|\theta^{P,n+1}_{i+1/2}\right|\right)
+ \bigg(1 - \frac{t - t_n}{\Delta t}\bigg) \, \sum_{i\in\I} \Delta x \, f\left(\left|\theta^{P,n}_{i+1/2}\right|\right).
\end{multline*}
Thanks to estimate~\eqref{3gradient_entropy_estimate} we get
\begin{align*}
\int_{I_P} f\left(\left|\partial_x u^{P, \eps}_M(x,t) \right|\right)\, \dd x &\leq\ln(2)\, \|\partial_x u_0\|_{L^1(\R)} + 2 e^{-1} +2 \sum_{i \in \I} \Delta x f\left(\left|\theta^{P, 0}_{i+1/2}\right|\right) + C\, T.
\end{align*}
Applying Lemma \ref{3lem.boundinitent}, we deduce that 
\begin{align*}
   \int_{I_P} f\left(\left|\partial_x u^{P, \eps}_M(x,t)\right| \right)\, \dd x 
   \leq \ln(2)\,  \|\partial_x u_0\|_{L^1(\R)} + 2 e^{-1} +2C_0 + C\, T.
\end{align*}
Therefore, by Lemma~\ref{3Llog L estimate}, we get
\begin{align*}
    \|\partial_x u^{P, \eps}_M\|_{L^{\infty}(0, T; L\log L(I_P))} \le 1 + \|\partial_x u^{P, \eps}_M\|_{L^{\infty}(0, T; L^1(I_P))}\ln(1 + e^2) + \sup_{t \in (0, T)}\int_{I_P} f(|\partial_x u^{P, \eps}_M |)\, \dd x,
\end{align*}
that concludes the desired estimate.

\subsection{Compactness properties}\label{3discontinuous viscosity solution} In order to study the convergence of $u^{P, \varepsilon}_M$, we take advantage of the results established in Theorem~\ref{3theorem-properties} and of Simon’s compactness lemma.
\begin{lemma}[Simon's Lemma {\cite[Lemma 4.4]{1}}]\label{3Simon's lemma}
Let $X$, $B$ and $Y$ be Banach spaces such that
\[
X \hookrightarrow B \quad \text{compactly},
\qquad
B \hookrightarrow Y \quad \text{continuously}.
\]
Let $T>0$ and let $(u_n)$ be a sequence satisfying
\[
(u_n) \text{ bounded in } L^\infty(0,T;X),
\]
and
\[
(\partial_t u_n) \text{ bounded in } L^r(0,T;Y)
\quad \text{for some } r>1.
\]
Then $(u_n)$ is relatively compact in
\[
C([0,T];B).
\]
\end{lemma}
\begin{lemma}\label{3uniform_regularity}
Let $T > 0$ and let $I_P \subset \mathbb{R}$ be defined in \eqref{3defI}. Consider a family of periodic functions 
$(u^{P,\varepsilon}_M)_{\varepsilon > 0}$ defined on $I_P \times [0,T]$, satisfying the uniform bounds~\eqref{3first_L_infty_estimate}--\eqref{3bounded in LlogL}. Then, the family $(u^{P,\varepsilon}_M)_{\varepsilon}$ is relatively compact in $C([0,T];L^{\infty}(I_P))$. In particular, up to the extraction of a subsequence, there exists a function
\[
u^P_M \in C([0,T];L^{\infty}(I_P))
\]
such that
\[
u^{P,\varepsilon}_M \to u^P_M
\quad \text{in } C([0,T];L^{\infty}(I_P)),
\]
that is,
\[
\lim_{\varepsilon\to0}
\|u^{P,\varepsilon}_M(\cdot,t)
-
u^P_M(\cdot,t)\|_{L^{\infty}(I_P \times (0, T))}
=0.
\]
\end{lemma}
\begin{proof} 
For each fixed $t\in[0,T]$, the estimates 
\eqref{3first_L_infty_estimate}--\eqref{3bounded ux}--\eqref{3bounded in LlogL} imply that
\[
u^{P,\varepsilon}_M \in L^\infty(0,T; W^{1,L\log L}(I_P)).
\]
Moreover, thanks to \eqref{3time estimate of u}, we have
\[
\partial_t u^{P,\varepsilon}_M
\text{ bounded in }
L^\infty(0,T;L^1(I_P)).
\]
By Lemma~\ref{3compact} and by applying Lemma~\ref{3Simon's lemma}
with
\[
X = W^{1,L\log L}(I_P), 
\qquad 
B = L^\infty(I_P), 
\qquad 
Y = L^1(I_P),
\]
and using the fact that
\[
L^\infty(I_P) \hookrightarrow L^1(I_P)
\]
continuously (since $I_P$ is bounded), we deduce that
$u^{P,\varepsilon}_M$ is relatively compact in
$C([0,T];L^\infty(I_P)).$ Consequently, there exist a subsequence
and a limit function $u^P_M$ such that
\[
u^{P,\varepsilon}_M \to u^P_M
\quad \text{in } C([0,T];L^\infty(I_P)).
\]
This completes the proof.
\end{proof}

\subsection{Further properties of the scheme and proof of Theorem~\ref{3theorem-convergence}}\label{3viscosity solution}

In this section, we prove Theorem~\ref{3theorem-convergence}. We divide our proof into two steps. We first introduce the notions of stability, monotonicity, and consistency for our numerical scheme, and then we prove that scheme~\eqref{3sch:eq} satisfies these properties. Then, we prove the convergence of the scheme. But first, let us define the following operator:
\begin{multline} \label{3eq:operator_node}
S_{\varepsilon}(x_i, t_{n+1}, u^{P, n + 1}, u^{P, \varepsilon}_M) :=
\frac{u^{P, n + 1}_i - \frac12 \big( u^{P, \varepsilon}_M(x_i+\Delta x, t_{n+1}-\Delta t) + u^{P, \varepsilon}_M(x_i, t_{n+1}-\Delta t) \big)}{\Delta t}
\\
- \lambda_i[u^{P, n + 1}]\, \bigg[\frac{ \big| u^{P, \varepsilon}_M(x_i+\Delta x, t_{n+1}-\Delta t) - u^{P, \varepsilon}_M(x_i, t_{n+1}-\Delta t) \big|}{\Delta x}\bigg],
\end{multline}
where the discrete convolution is given by \eqref{3def:lambda}. We notice then that scheme \eqref{3sch:eq} is equivalent to
\begin{equation} \label{3eq:scheme_equiv_node}
S_{\varepsilon}(x_i, t_{n+1}, u^{P, n+1}, u^{P, \varepsilon}_M) = 0
\qquad\text{for every }i \in \I,\,n \in \{0, ..., N_T\}.
\end{equation}
\begin{definition}[Monotonicity, Stability, and Consistency] \label{3def:viscosit}
We say that the numerical scheme \eqref{3sch:eq} is:

\begin{itemize}
\item \textbf{Monotone:} if and only if, for every $v=(v_j)_{j\in\I}$ with $|v_j|\le\alpha$,
for all $(x,t)\in\R\times(0,T)$, and for all
$\psi_1,\psi_2 \in
[-\|u_0^P\|_{L^\infty(I_P)},\,\|u_0^P\|_{L^\infty(I_P)}]$
such that $\psi_1\ge\psi_2$, one has
\[
S_{\varepsilon}(x,t,v,\psi_1)
\le
S_{\varepsilon}(x,t,v,\psi_2).
\]
\item \textbf{Stable:} if and only if $u^{P, \varepsilon}_M,$ defined in \eqref{3defQ1ext}, is bounded independently of $\varepsilon$.

\item \textbf{Consistent:} 
for every $\varphi\in C^1(I_P\times[0,T])$, every grid point
$(x_i,t_{n+1})$ with $i\in\I$ and $n\in\{0,\ldots,N_T\}$,
every $(x,t)\in I_P\times(0,T]$, and every $\xi\in\mathbb{R}$, then for
\[
\varphi(\cdot,t_{n+1})
:=
\bigl(\varphi(x_j,t_{n+1})\bigr)_{j\in\I},
\]
we have
\[
\lim_{\substack{
(x_i,t_{n+1})\to(x,t)\\
\varepsilon\to0\\
\xi\to0}}
S_{\varepsilon}\bigl(
x_i,t_{n+1},
\varphi(\cdot,t_{n+1})+\xi,
\varphi+\xi
\bigr)
=
\partial_t\varphi(x,t)
-
\Bigl[
\bigl(\sigma_M^P(\cdot)\star\varphi(\cdot,t)\bigr)(x)
+a(t)
\Bigr]
\bigl|\partial_x\varphi(x,t)\bigr|.
\]
\end{itemize}
\end{definition}
\begin{proposition}
The scheme \eqref{3sch:eq} is stable, monotone and consistent in the sense of Definition~\ref{3def:viscosit}.
\end{proposition}
\begin{proof} To prove that scheme \eqref{3sch:eq} is stable, monotone, and consistent, we proceed in three steps, dealing with each property separately: first stability, then monotonicity, and finally consistency.\\
\textbf{Step 1.} (Stability):
From \eqref{3ineq:Linfty} given in Theorem \ref{3theorem-properties}, we have
\[
\max_{i \in \I} |u^{P,n+1}_i| \le 2\|u_0\|_{L^\infty(\R)},
\]
for all $1 \le n \le N_T$. Therefore, we deduce that $u^{P, \varepsilon}_M$ defined in \eqref{3defQ1ext} satisfies a uniform $L^\infty$-bound, which proves the stability of the scheme.\\
    \textbf{Step 2.} (Monotonicity):
    We prove that the scheme is monotone in the sense of Definition \ref{3def:viscosit}.  
For simplicity, fix $(x_i,t_{n+1})$ and denote
\[
\psi_i := \psi(x_i,t_{n+1}-\Delta t),
\qquad
\psi_{i+1} := \psi(x_i+\Delta x,t_{n+1}-\Delta t).
\]
From \eqref{3eq:operator_node}, the operator
$S_{\varepsilon}$ can be rewritten as
\[
S_{\varepsilon}(x_i,t_{n+1},v,\psi)
=
\frac{v_i - T_{\varepsilon}(\psi)}{\Delta t},
\]
where
\[
T_{\varepsilon}(\psi)
=
\frac12(\psi_{i+1}+\psi_i)
+
\frac{\Delta t}{\Delta x}
\lambda_i[v] \,|\psi_{i+1}-\psi_i|.
\]
Therefore, to prove that $S_{\varepsilon}$ is monotone, it suffices to show that
$T_{\varepsilon}$ is non-decreasing with respect to
$\psi_{i+1}$ and $\psi_i$. We compute the partial derivatives
\[
\frac{\partial T_{\varepsilon}}{\partial \psi_{i+1}}
=
\frac12
+
\frac{\Delta t}{\Delta x}
\lambda_i[v]
\,\mathrm{sign}(\psi_{i+1}-\psi_i),
\]
and
\[
\frac{\partial T_{\varepsilon}}{\partial \psi_i}
=
\frac12
-
\frac{\Delta t}{\Delta x}
\lambda_i[v]
\,\mathrm{sign}(\psi_{i+1}-\psi_i).
\]
with respect to $\psi_{i+1}$, $\psi_i$ respectively.
By the CFL condition \eqref{3deltat/deltax}, both partial derivatives are nonnegative:
\[
\frac{\partial T_{\varepsilon}}{\partial \psi_{i+1}} \ge 0,
\qquad
\frac{\partial T_{\varepsilon}}{\partial \psi_i} \ge 0.
\]
Hence $T_{\varepsilon}$ is non-decreasing with respect to
$\psi_{i+1}$ and $\psi_i$. Consequently, for every
$\psi_1 \ge \psi_2$,
\[
T_{\varepsilon}(\psi_1)
\ge
T_{\varepsilon}(\psi_2),
\]
which implies
\[
S_{\varepsilon}(x_i, t_{n + 1}, v,\psi_1)
\le
S_{\varepsilon}(x_i, t_{n + 1}, v,\psi_2).
\]
Therefore, the scheme is monotone in the sense of Definition \ref{3def:viscosit}.\\
\textbf{Step 3.} (Consistency):
Let $\varphi \in C^1(I_P\times[0,T])$,
let $(x,t)\in I_P\times(0,T]$,
and let $\xi\in\mathbb{R}$.
Let $(x_i,t_{n+1})$ be grid points such that
\[
(x_i,t_{n+1}) \longrightarrow (x,t)
\quad \text{as } \Delta x,\Delta t \to 0.
\]We evaluate the operator $S_{\varepsilon}$
defined in \eqref{3eq:operator_node}
at the grid node $(x_i,t_{n+1})$:

\begin{multline*}
 S_{\varepsilon}\big(x_i,t_{n+1},
\varphi(\cdot,t_{n+1})+\xi,
\varphi+\xi\big)
=
\frac{\varphi(x_i,t_{n+1})+\xi
-\frac12\Big(
\varphi(x_{i+1},t_n)+\xi
+
\varphi(x_i,t_n)+\xi
\Big)}
{\Delta t}
\\
-
\lambda_i[\varphi(\cdot,t_{n+1}) +\xi]
\left|
\frac{
\varphi(x_{i+1},t_n)
-
\varphi(x_i,t_n)
}{\Delta x}
\right|.
\end{multline*}
Since the constant $\xi$ cancels in the numerator, we obtain
\begin{multline*}
 S_{\varepsilon}\big(x_i,t_{n+1},
\varphi(\cdot,t_{n+1})+\xi,
\varphi+\xi\big)
=
\frac{\varphi(x_i,t_{n+1})
-\frac12\Big(
\varphi(x_{i+1},t_n)
+
\varphi(x_i,t_n)
\Big)}
{\Delta t}
\\
-
\lambda_i[\varphi(\cdot,t_{n+1}) +\xi]
\left|
\frac{
\varphi(x_{i+1},t_n)
-
\varphi(x_i,t_n)
}{\Delta x}
\right|.
\end{multline*}
Since $\varphi\in C^1$, Taylor expansion in time gives
\[
\varphi(x_i,t_n)
=
\varphi(x_i,t_{n+1})
-
\Delta t\,\partial_t \varphi(x_i,t_{n+1})
+
o(\Delta t).
\]
Moreover, expanding in space,
\[
\varphi(x_{i+1},t_n)
=
\varphi(x_i,t_n)
+
\Delta x\,\partial_x \varphi(x_i,t_n)
+
o(\Delta x).
\]Combining both expansions, we obtain
\[
\frac{\varphi(x_i,t_{n+1})
-\frac12\Big(
\varphi(x_{i+1},t_n)
+
\varphi(x_i,t_n)
\Big)}
{\Delta t}
=
\partial_t \varphi(x_i,t_{n+1})
+
o(1),
\]
as $\Delta t,\Delta x \to 0$. For the second term, we also have
\[
\frac{
\varphi(x_{i+1},t_n)
-
\varphi(x_i,t_n)
}{\Delta x}
=
\partial_x \varphi(x_i,t_n)
+
o(1),
\]
and therefore
\[
\left|
\frac{
\varphi(x_{i+1},t_n)
-
\varphi(x_i,t_n)
}{\Delta x}
\right|
=
|\partial_x \varphi(x_i,t_{n+1})|\,
+ \,
o(1).
\]
Regarding the discrete convolution term, we have from \eqref{3def:lambda},
\[
\lambda_i[\varphi (\cdot,t_{n+1})+\xi]
= a(t_{n + 1}) + \sum_{j\in\I}
\Delta x \,
\sigma^P_{M,j}\,
\varphi(x_{i-j}, t_{n+1}) \;+\; \xi\sum_{j\in\I} \Delta x\,\sigma^P_{M,j}.
\]
The term $\displaystyle{\sum_{j \in \I}\Delta x\,\sigma^P_{M,j}}$ is uniformly bounded by \eqref{3Normcesaro}, which implies that $\displaystyle{\xi\sum_{j \in \I}\Delta x\,\sigma^P_{M,j} \to 0}$ as $\xi \to 0.$
Moreover, the remaining Riemann sum
\(\displaystyle{\sum_{j\in\I}
\Delta x \,
\sigma^P_{M}(x_j)\,
\varphi(x_{i-j}, t_{n+1})}\) converges to the continuous convolution
\[
(\sigma^P_M(\cdot)\star \varphi(\cdot,t))(x) = \int_{I_P}\sigma^P_M(z)\,
\varphi(t,x-z)\,dz.,
\]
as $\varepsilon=(\Delta x,\Delta t)\to (0,0)$. Therefore, for any sequence of grid
nodes $(x_i,t_{n+1})\to(x,t)$ as $\varepsilon=(\Delta x,\Delta t)\to (0,0)$ and for $\xi\to 0$,
\[
\begin{aligned}
\lim_{\substack{(x_i, t_{n + 1})\to(x,t)\\ \varepsilon\to0\\ \xi\to0}}
S_{\varepsilon}\big(x_i,t_{n+1},
\varphi(\cdot,t_{n+1})+\xi,
\varphi+\xi\big)
&= \partial_t\varphi(x,t) - \left[\big(\sigma^P_M (\cdot) \star\varphi(\cdot,t)\big)(x)  + a(t)\right]\,|\partial_x\varphi(x,t)|.
\end{aligned}
\]
Therefore, the scheme is consistent in the sense of
Definition~\ref{3def:viscosit}.
\end{proof}

Before to conclude the proof of Theorem~\ref{3theorem-convergence}, we recall the definition of continuous viscosity solutions for~\eqref{31.equPaux}. For further background on viscosity solutions, we refer the reader to Barles~\cite{5}, Crandall and Ishii~\cite{14}, and Crandall and Lions~\cite{15}.

\begin{definition}[Viscosity solution]
\label{3def:discont_visc_fixed}
Assume the convolution part is locally bounded on $I_P\times(0,T)$ and $u^P_0$ is locally bounded on $I_P$. 
We say that a continuous function $u^P_M\in C(I_P\times[0,T])$ is a viscosity solution of \eqref{31.equPaux} if:

\begin{enumerate}
\item \emph{Initial condition:} $u^P_M(x,0)=u^P_0(x)\quad\text{for all }x\in I_P$.

\item \emph{Test-function condition:}
\begin{itemize}
\item For any $\varphi\in C^1( I_P \times(0,T))$, if $u^P_M-\varphi$ attains a local maximum at $(x_0,t_0)\in I_P \times(0,T)$, then
\[
\partial_t\varphi(x_0,t_0)
-\left[\big(\sigma_M^P(\cdot)\star \varphi(\cdot,t_0)\big)(x_0) + a(t_0) \right]\,|\partial_x\varphi(x_0,t_0)|
\le 0.
\]

\item For any $\varphi\in C^1(I_P\times(0,T))$, if $u^P_M-\varphi$ attains a local minimum at $(x_0,t_0)$, then
\[
\partial_t\varphi(x_0,t_0)
-\left[\big(\sigma_M^P(\cdot)\star \varphi(\cdot,t_0)\big)(x_0) +  a(t_0) \right]\,|\partial_x\varphi(x_0,t_0)|
\ge 0.
\]
\end{itemize}
\end{enumerate}
\end{definition}

Now, we can proceed to the proof of Theorem \ref{3theorem-convergence}.

\begin{proof}[Proof of Theorem~\ref{3theorem-convergence}] We aim to show that $u^P_M(x,0) = u^P_0(x)$ for all $x \in I_P$. By Lemma \ref{3uniform_regularity}, we have the strong convergence of the family $(u^{P, \varepsilon}_M)_{\varepsilon}$ to the limit $u^P_M$
\[
u^{P, \varepsilon}_M \to u^P_M \quad \text{in } C([0,T]; L^\infty(I_P)).
\]
This uniform convergence in space and time implies that for any $t \in [0,T]$
\[
\lim_{\varepsilon \to 0} \| u^{P, \varepsilon}_M(\cdot, t) - u^P_M(\cdot, t) \|_{L^\infty(I_P)} = 0.
\]
In particular, evaluating this at the initial time $t=0$, we obtain
\begin{equation}\label{3eq:limit_t0}
\lim_{\varepsilon \to 0} \| u^{P, \varepsilon}_M(\cdot, 0) - u^P_M(\cdot, 0) \|_{L^\infty(I_P)} = 0.
\end{equation}
At $t=0$, we consider the first time interval where $n=0$ and $t_n = 0$. Substituting $t=0$ into \eqref{3defQ1ext}, we obtain
\[
u^{P, \varepsilon}_M(x, 0) = \left( \frac{x - x_i}{\Delta x} \right) u^{P, 0}_{i+1} + \left( 1 - \frac{x - x_i}{\Delta x} \right) u^{P, 0}_i.
\]
Using the definition of the discrete initial data $u^{P, 0}_j = u^P_0(x_j)$ from the scheme initialization, we have
\begin{equation}\label{3eq:interpolant}
u^{P, \varepsilon}_M(x, 0) = \left(\frac{x - x_i}{\Delta x}\right) u^P_0(x_{i+1}) + \left( 1 - \frac{x - x_i}{\Delta x} \right) u^P_0(x_i), \quad \text{for } x \in [x_i, x_{i+1}].
\end{equation}
Equation \eqref{3eq:interpolant} shows that $u^{P, \varepsilon}_M(\cdot, 0)$ is precisely the piecewise linear interpolant of the initial function $u^P_0$ on the spatial grid. Since $u^P_0$ is continuous on the interval $I_P$, the error of linear interpolation vanishes as the mesh size $\Delta x \to 0$ (which corresponds to $\varepsilon \to 0$):
\begin{equation}\label{3eq:interp_conv}
\| u^{P, \varepsilon}_M (\cdot, 0) - u^P_0 \|_{L^\infty(I_P)} \to 0 \quad \text{as } \varepsilon \to 0.
\end{equation}
For any $x \in I_P$, we apply the triangle inequality:
\[
| u^P_M(x, 0) - u^P_0(x) | \le | u^P_M(x, 0) - u^{P, \varepsilon}_M(x, 0) |\, + \, | u^{P, \varepsilon}_M(x, 0) - u^P_0(x) |.
\]
Taking the supremum over $x \in I_P$ yields:
\[
\| u^P_M(\cdot, 0) - u^P_0 \|_{L^\infty(I_P)} \le \| u^P_M(\cdot, 0) - u^{P, \varepsilon}_M(\cdot, 0) \|_{L^\infty(I_P)} + \, \| u^{P, \varepsilon}_M (\cdot, 0) - u^P_0 \|_{L^\infty(I_P)}.
\]
As $\varepsilon \to 0$, the first term on the right-hand side converges to $0$ by \eqref{3eq:limit_t0}. Moreover, the second term converges to $0$ due to the consistency of the spatial interpolation \eqref{3eq:interp_conv}. Consequently,
\[
\| u^P_M(\cdot, 0) - u^P_0 \|_{L^\infty(I_P)} = 0,
\]
which implies that
\[
u^P_M(x, 0) = u^P_0(x) \qquad \text{for all } x \in I_P.
\] 
 We are now ready to prove that $u^{P}_M$ is a viscosity solution. 
Let $(x_0,t_0)\in I_P\times(0,T]$ and let $\varphi\in C^1(I_P\times[0,T])$ be such that 
$u^P_M - \varphi$ attains a strict local maximum at $(x_0,t_0)$. 
By a standard argument, we may assume without loss of generality that $\varphi$ is tangent from above to $u^P_M$ at $(x_0,t_0)\in I_P\times(0,T)$. 
We must show that
\[
\partial_t\varphi(x_0,t_0)
- \left[\big(\sigma^P_M\star\varphi(\cdot,t_0)\big)(x_0) + a(t_0)\right]\,|\partial_x\varphi(x_0,t_0)|
\le 0.
\]
By a standard technique used in the theory of viscosity solutions (see Lemma \ref{3barles}), we can say that there exists a subsequence still denoted as
\[
(\varepsilon_m, x_{m}, t_{m}) \to (0, x_0, t_0)
\quad \text{as } m \to +\infty,
\]
such that $(x_{m}, t_{m})$ is a local maximum of $u^{P, \varepsilon_m}_M - \varphi$ and
\[u^P_M(x_0, t_0) = \lim_{m\to+\infty} u^{P, \varepsilon_m}_M(x_{m}, t_{m}).
\]
We define $\xi_m$ by
\[
\xi_m = u^{P,\varepsilon_m}_M(x_m,t_m) - \varphi(x_m,t_m).
\]
Since $(x_m,t_m)$ is a local maximum point of 
$u^{P,\varepsilon_m}_M - \varphi$, there exists a neighborhood 
$\mathcal{U}_m$ of $(x_m,t_m)$ such that
\[
(u^{P,\varepsilon_m}_M - \varphi)(x,t)
\le
(u^{P,\varepsilon_m}_M - \varphi)(x_m,t_m)
=
\xi_m,
\qquad \forall (x,t)\in \mathcal{U}_m.
\]
In particular,
\[
u^{P,\varepsilon_m}_M(x_m,t_m)
=
\varphi(x_m,t_m)+\xi_m,
\]
and, for all $(x,t)$ in a neighborhood of $(x_m,t_m)$,
\[
u^{P,\varepsilon_m}_M(x,t)
\le
\varphi(x,t)+\xi_m.
\]
Then using the monotonicity of the operator $S_{\varepsilon_m}$, we get
\begin{align*}
    S_{\varepsilon_m}\!\left(x_{m},
t_{m}, \varphi(\cdot, t_{m}) + \xi_m, \varphi + \xi_m
\right)
\le 
&S_{\varepsilon_m}\!\left(x_{m}, 
t_{m}, \varphi(\cdot, t_{m}) + \xi_m, u^{P, \varepsilon_m}_M
\right)
\\
&= S_{\varepsilon_m}\!\left( x_{m}, 
t_{m}, u^{P, \varepsilon_m}_M(\cdot, t_{m}), u^{P, \varepsilon_m}_M\right)
= 0,
\end{align*}
using \eqref{3eq:scheme_equiv_node}. Now, since $\varphi_m = \varphi(x_m,t_m)$ and $\varphi$ is continuous, we have
\[
\lim_{m\to+\infty} \varphi_m
=
\varphi(x_0,t_0)
=
u^P_M(x_0,t_0).
\]
Moreover, using Young's inequality and the convergence
$u^{P,\varepsilon}_M \to u^P_M$ uniformly on $I_P\times[0,T]$, we obtain
\begin{equation} \label{3eq:conv_L1_convergence}
\sup_{t \in [0,T]} 
\Big\| \sigma^P_M \star u^{P,\varepsilon}_M(\cdot, t) 
       - \sigma^P_M \star u^P_M(\cdot, t) \Big\|_{L^{\infty}(I_P)}
\le 
\|\sigma^P_M\|_{L^1(I_P)} \; 
\sup_{t \in [0,T]} \big\| u^{P,\varepsilon}_M(\cdot, t) - u^P_M(\cdot, t) \big\|_{L^{\infty}(I_P)}
\;\xrightarrow[\varepsilon \to 0]{}\; 0,
\end{equation}
By the consistency of the scheme $S_{\varepsilon}$, we deduce
\begin{align*}
&\partial_t\varphi(x_0,t_0)
-
\left[\left(\sigma^P_M \star \varphi(\cdot,t_0)\right)(x_0) + a(t_0)\right]
\,|\partial_x\varphi(x_0,t_0)| \\
&\qquad =
\lim_{\substack{(x_i, t_{n + 1})\to(x_0,t_0)\\ \varepsilon\to0\\ \xi\to0}}
S_{\varepsilon}\big(x_i,t_{n+1},
\varphi(\cdot,t_{n+1})+\xi,
\varphi+\xi\big)\\
&\qquad =
\lim_{m\to+\infty}
S_{\varepsilon_m}\!\left(
x_m,t_m,
\varphi(\cdot,t_m)+\xi_m,
\varphi+\xi_m
\right)\\
& \qquad \le \lim_{\substack{m \to + \infty} }S_{\varepsilon_m}\!\left( x_{m}, 
t_{m}, u^{P, \varepsilon_m}_M(\cdot, t_{m}), u^{P, \varepsilon_m}_M\right)
= 0.
\end{align*}
Similarly, if $u^P_M - \varphi$ attains a local minimum at $(x_0,t_0)$, we obtain
\[
\partial_t\varphi(x_0,t_0)
-
\left[\big(\sigma^P_M \star \varphi(\cdot,t_0)\big)(x_0) + a(t_0)\right]
\,|\partial_x\varphi(x_0,t_0)|
\ge 0.
\]
Therefore, $u^P_M$ is a viscosity solution of \eqref{31.equPaux}. This completes the proof of Theorem~\ref{3theorem-convergence}.
\end{proof}

\subsection{Proof of Theorem~\ref{3theorem-limit-PM}}\label{3existence of viscosity solution} 
In this section, we prove the convergence toward the initial model~\ref{3eq:eikonal}. The proof proceeds in two steps. We first pass to the limit as $M\to+\infty$ and then as $P\to+\infty$. By Fejér's theorem,
\[
\sigma^P_M \to \K^P
\quad \text{in } L^1(I_P).
\]
The estimates
\eqref{3first_L_infty_estimate}--\eqref{3bounded in LlogL}
are uniform in $M$. Hence the family $(u^P_M)_M$ is relatively compact in the space
$C([0,T];L^\infty(I_P))$, and, up to a subsequence,
\[
u^P_M \to u^P
\quad \text{locally uniformly on } I_P\times[0,T].
\]
By continuity of convolution from $L^1 \times L^\infty$ to $L^\infty$,
\[
\sigma^P_M \ast u^P_M
\longrightarrow
\K^P \ast u^P
\quad \text{locally uniformly}.
\]
Since each $u^P_M$ is a viscosity solution, the stability theorem for
continuous viscosity solutions
(see Barles~\cite[Theorem~2.3]{barles1993})
implies that the limit $u^P$ is a viscosity solution of \eqref{32.periodic}.\\
Then, for the second step, we pass to the limit as $P \to + \infty$.
We observe that
\[
\|u^P_0\|_{L^\infty(I_P)}
\le
2\|u_0\|_{L^\infty(\mathbb R)},
\qquad
\|\partial_x u^P_0\|_{L^1(I_P)}\le \| \partial_x u_0\|_{L^1(\R)}.
\]
Therefore, the estimates
\eqref{3first_L_infty_estimate}--\eqref{3bounded in LlogL}
are uniform in $P,M,\varepsilon$.
Consequently, $(u^P)_P$ is uniformly bounded in
\[
L^\infty((0,T);W^{1,L\log L}(I_P)),
\]
and $(\partial_t u^P)_P$ is uniformly bounded in
\[ L^{\infty}((0,T);L^1(I_P)).\]
By Simon's compactness lemma (see also Lemma~\ref{3Simon's lemma}),
$(u^P)_P$ is relatively compact in
$C([0,T];L^\infty(I_P))$.
Hence, up to a subsequence, we have
\[
u^P \to u
\quad \text{locally uniformly on } \mathbb{R}\times[0,T].
\]
Applying again the stability theorem for viscosity solutions,
we conclude that the limit $u$ is a viscosity solution of \eqref{3eq:eikonal}.

\section{Numerical experiment}\label{3numerical experiment}
This section is devoted to the numerical simulations of the solution $u^P_M$ to equation \eqref{31.equPaux}. As established in the theoretical part of this work, the problem admits a continuous viscosity solution. The objective of this section is to verify that the numerical simulations are consistent with this theoretical framework and to illustrate the qualitative behavior of the solution. The section is divided into two subsections. The first presents numerical simulations in the presence of an external stress, highlighting its influence on the dynamics of the solution. The second considers the case without external stress, allowing us to examine the intrinsic evolution of the system and to compare it with the stressed case. 

\subsection{Numerical simulations with external stress}\label{3external stress} In this subsection, we investigate the evolution of $u^P_M$ in the presence of an external stress. To do this, we consider the following periodic initial data
\begin{equation*}
  u_0(x)=  u^P_0(x)=  \cos\left(\dfrac{x}{20}\right) + 1, \qquad x \in [-P, P],
\end{equation*}
and we extend $u^P_0$ by $P$ periodicity on $\R$.  Moreover, we consider a physical kernel that naturally appears in the so-called Peierls-Nabarro model, which is known in the framework of isotropic elasticity with edge dislocations (see \cite{alvarez2006dislocation, 19} for more details). This kernel is given by 
$$
\K(x) =  \frac{\mu b^2}{2\pi (1-\nu)} \frac{x^2-\zeta^2}{(x^2+\zeta^2)^2}, 
$$
where $\nu=\frac{\lambda}{2(\lambda+\mu)}$  is the Poisson ratio and $\lambda, \mu>0$ are the Lamé coefficients for isotropic elasticity. Here, $\zeta \neq 0$  is a physical parameter depending only on the material and represents the size of the dislocation core. Moreover, the vector 
$\vec{b}=b(1,0)$ is the Burgers vector, which reflects the direction of dislocation motion. 
In our simulations, we took the special case $\frac{\mu b^2}{2\pi (1-\nu)} =1$ and $\zeta=1$ for simplicity. Obviously, these initial data and kernel verify the required assumptions {\bf (H1)-(H2)}. \\
Taking the numerical parameters indicated in Table~\ref{3table}, we observe the evolution of the solution in Figure~\ref{3sol:external stress}.
\begin{table}[!h]
 \begin{tabular}{|*{5}{c|}}
     \hline
  \quad $M$ \quad & \quad $N$ \quad & \quad $P$ \quad  &  \quad $\Delta t$ \quad \\ \hline 
    \quad $400$ \quad &  \quad $500$  \quad & \quad $50$ \quad & \quad $0.002$ \quad \\ \hline
   \end{tabular}
        \caption{Used parameters in Figures ~\ref{3sol:external stress}-~\ref{3stress: neg total variation}}\label{3table}
\end{table}
   
   \begin{figure}
       \centering
       \includegraphics[width=0.5\linewidth]{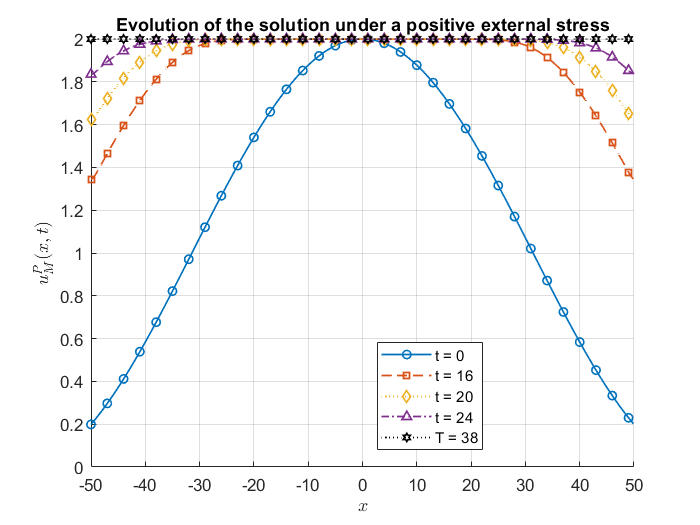}
       \caption{Evolution of $u^P_M$ under a positive external stress.}
       \label{3sol:external stress}
   \end{figure}
     \begin{figure}
       \centering
       \includegraphics[width=0.5\linewidth]{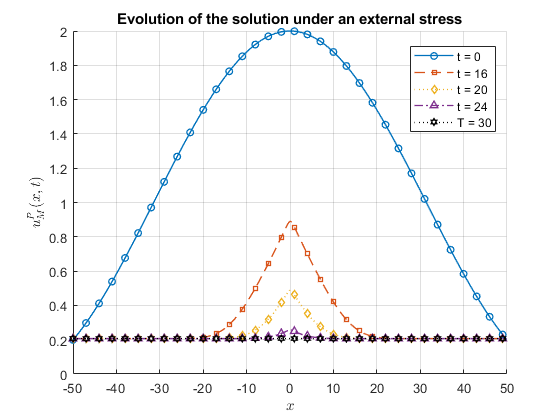}
       \caption{Evolution of $u^P_M$ under a negative external stress.}
       \label{3sol: neg external stress}
   \end{figure}
As shown in Figure~\ref{3sol:external stress}, which corresponds to the case of a positive external stress, $a(t) = 2$, the solution $u^P_M$ initially exhibits significant spatial variations. These variations correspond to regions where the dislocations are highly concentrated. As time evolves, the solution progressively smooths out and approaches an almost constant state at $T=38$, whose value is the maximum of the initial data. This behavior indicates that the system evolves toward a stable equilibrium configuration, which is simply the stationary solution of the Eikonal equation $\partial_t u = 2 |\partial_x u|$. 

In contrast, when the external stress is negative, namely $a(t)=-2$, as illustrated in Figure~\ref{3sol: neg external stress}, the solution exhibits a similar qualitative behavior. It gradually becomes spatially uniform, but the limiting constant is now the minimum of the initial data rather than its maximum, which reflects the stationary solution of the Eikonal equation $\partial_t u = -2 |\partial_x u|$.

Figure~\ref{3der: discrete spatial derivative} illustrates the evolution of the discrete spatial derivative $\partial_x u^P_M$ under a positive external stress. Initially, the profile contains both positive and negative regions, corresponding to dislocations of opposite sign. During the early stage of the evolution, these dislocation densities become more localized near the boundaries of the domain, leading to the formation of sharper peaks, given that periodic boundary conditions are assumed. Subsequently, the amplitudes of both the positive and negative regions gradually decrease as the positive and negative dislocations progressively compensate each other through the dislocation annihilation process. Throughout the evolution, the total dislocation is conserved and remains equal to zero. As a result, the discrete spatial derivative converges to zero at long times (here $T=38$), indicating that the system eventually relaxes toward a dislocation-free equilibrium state.
\begin{figure}
    \centering
    \includegraphics[width=0.4\linewidth]{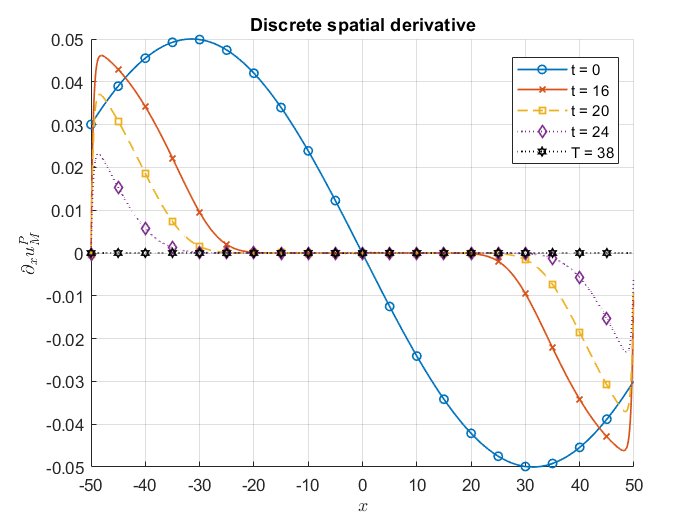}
    \caption{Evolution of $\partial_x u^P_M$ under a positive external stress.}
    \label{3der: discrete spatial derivative}
\end{figure}

Figure~\ref{3der: neg discrete spatial derivative} illustrates the evolution of the discrete spatial derivative $\partial_x u^P_M$ under a negative external stress. As in the positive-stress case, the initial profile contains positive and negative dislocation densities. The main difference appears at the beginning of the evolution: the positive and negative densities become more localized near the origin, producing sharper peaks before their amplitudes start to decrease. Subsequently, the positive and negative contributions progressively compensate each other through the dislocation annihilation process. We thus observe behavior analogous to that of the previous case: the system relaxes to a dislocation-free equilibrium state at long times.

 \begin{figure}
    \centering
    \includegraphics[width=0.4\linewidth]{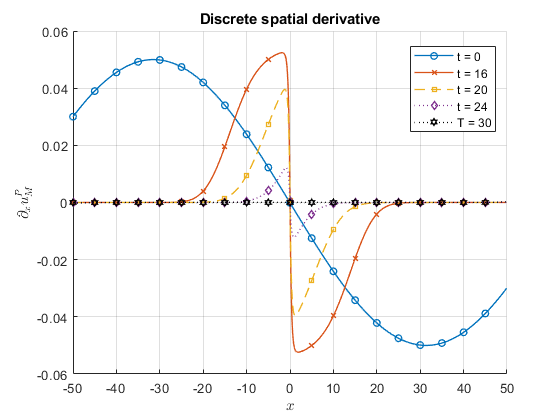}
    \caption{Evolution of $\partial_x u^P_M$ under a negative external stress.}
    \label{3der: neg discrete spatial derivative}
\end{figure}

This observation is corroborated by Figures~\ref{3stress: total variation} and~\ref{3stress: neg total variation}, which illustrate the evolution of the total variation of the numerical solution over time for the cases of positive and negative external stress, respectively. The total variation is initially high due to the steep gradients present in the initial data. It then decreases monotonically, eventually vanishing as expected.

\begin{figure}[ht!]
    \centering
    \includegraphics[width=0.4\linewidth]{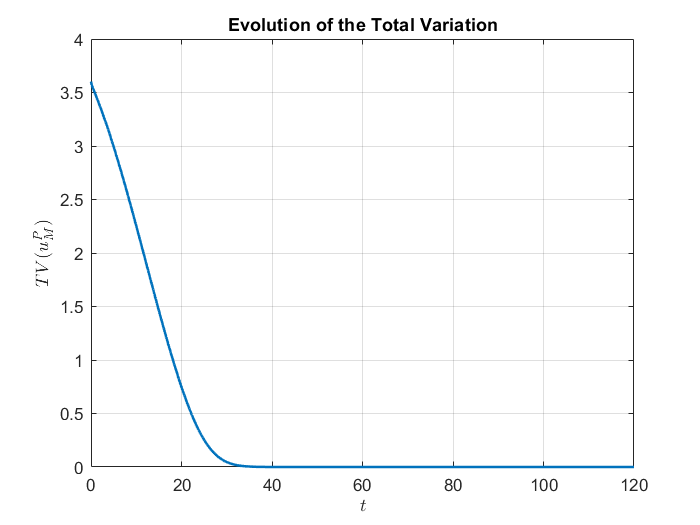}
    \caption{Total variation of $u^P_M$ under a positive external stress.}
    \label{3stress: total variation}
\end{figure}
\begin{figure}[ht!]
    \centering
    \includegraphics[width=0.4\linewidth]{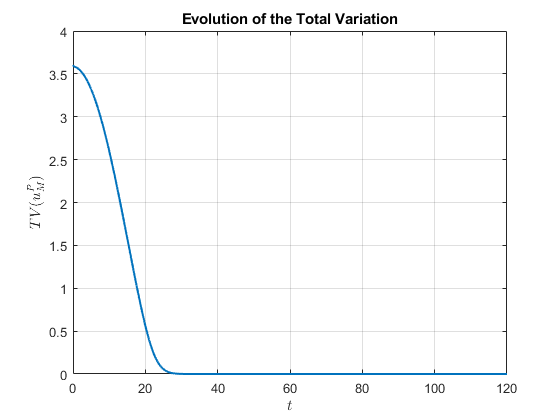}
    \caption{Total variation of $u^P_M$ under a negative external stress.}
    \label{3stress: neg total variation}
\end{figure}

It should be noted that similar behavior is also observed in the absence of external constraint, as shown in the following subsection. However, convergence toward the steady state is significantly slower. This result is expected, given that external constraint promotes evolution and thus accelerates convergence, whereas in its absence, the dynamics are governed solely by internal constraint, leading to a slower convergence rate.
\newpage 

\subsection{Numerical simulations without external stress}
To assess the role of the external stress, we consider the evolution of the dislocation lines under the sole effect of the internal stress. The numerical setup is identical to that of Subsection~\ref{3external stress}, except that the external stress is set to zero:
\[
a(t)=0.
\]

As illustrated in Figure~\ref{3stress: without, evolution of u}, the solution still converges to a stationary state, but on a much longer time scale: the equilibrium is reached only at $T=2500$, compared with $T=38$ when an external stress is applied. This shows that the internal stress alone is sufficient to drive the system toward equilibrium, while the external stress mainly accelerates the relaxation process.

\begin{figure}[ht!] \centering \includegraphics[width=0.5\linewidth]{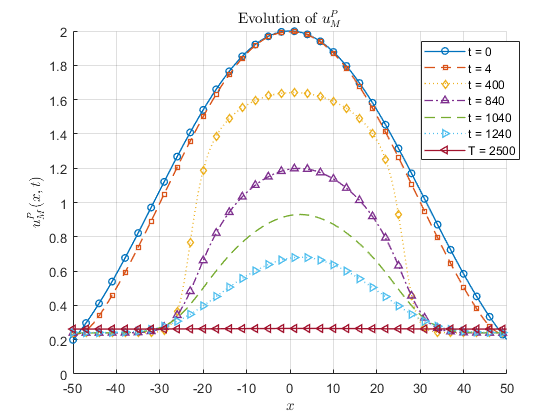} \caption{Evolution of $u^P_M$.} \label{3stress: without, evolution of u} \end{figure}

Figures~\ref{3stress: without, spatial derivative of u} and~\ref{3stress: without, total variation of u} provide further insight into this evolution. Figure~\ref{3stress: without, spatial derivative of u} shows that the approximate dislocation density $\partial_x u^P_M$ initially develops sharper positive and negative peaks before their amplitudes gradually decrease. Since the total dislocation is conserved and equal to zero, the dislocation density eventually vanishes, corresponding to a dislocation-free equilibrium state. Figure~\ref{3stress: without, total variation of u} confirms this behavior: the total variation decreases monotonically and stabilizes as the stationary regime is approached, although over a significantly longer time scale than in the presence of external stress.

\begin{figure}[ht!]
    \centering
    \includegraphics[width=0.4\linewidth]{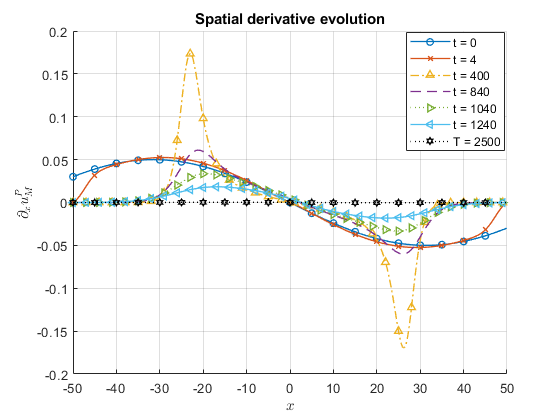}
    \caption{Spatial derivative of $u^P_M$.}
    \label{3stress: without, spatial derivative of u}
\end{figure}

\begin{figure}[ht!]
    \centering
    \includegraphics[width=0.4\linewidth]{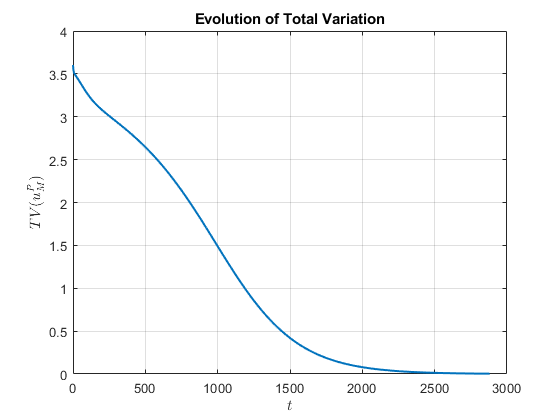}
    \caption{Total variation of $u^P_M$.}
    \label{3stress: without, total variation of u}
\end{figure}
Overall, the simulations show that both with and without external stress the system evolves toward the same equilibrium configuration. The main effect of the external stress is therefore not to modify the final state, but to substantially accelerate the convergence toward equilibrium. These observations are fully consistent with the physical interpretation of the model.

%%%%%%%%%%%%%%%%%%%%%%%%%%%%%%%%%%%%%%%%%%%%%%%%%%%%%%%%%%%%%%%%%%%%%%%%%%%%%%%
\begin{appendix}
\section{Technical results}\label{3app.techni}

\begin{lemma}[{\cite[Lemma 3]{ZHIZ25}}]\label{3lemma tech kernel}
Let assumption~\textbf{\emph{(H3)}} holds and assume that $\K \in L^1(\R)$. Then, $\K^P_M$ given by~\eqref{31.defKP} satisfies $\K^P_M \in L^1(I_P)$ with
\begin{align}\label{3normKPM}
    \|\mathcal{K}^P_M\|_{L^1(I_P)} &\leq 5 \|\mathcal{K}\|_{L^1(\mathbb{R})}.
\end{align}
Moreover, it holds
  \begin{align}\label{3Normcesaro}
     \sum_{j \in \I} \Delta x \, \left|\sigma_{M}^P(x_j)\right| \leq 5 \|\mathcal{K}\|_{L^1(\mathbb{R})}.
    \end{align}
\end{lemma}
\begin{lemma}[$L\log L$ estimate {\cite[Lemma 3.2]{HaMo10}} and {\cite[Lemma 4.1]{MoMo14}}]
\label{3Llog L estimate}
If $w \in L^1(I)$ is a nonnegative function, then $ \int_{I} f(w) \dd x <  \infty$ if and only if $w \in L\log L(I)$. Moreover, we have the following estimates
\begin{equation}\label{3f in llogl}
    \int_{I} f(w) \, \dd x \leq 1 + \|w\|_{L\,\log \,L(I)} + \|w\|_{L^1(I)} \, \ln\left(1 + \|w\|_{L\,\log \,L(I)}\right),
\end{equation}
and
\begin{equation}\label{3w in llogl}
    \|w\|_{L\,\log \,L (I)}\leq 1 + \|w\|_{L^1(I)}\,\ln(1 + e^2) +   \int_{I} f(w) \, \dd x.
\end{equation}
\end{lemma}
\begin{lemma}[Compact embedding {\cite[Theorem~8.32]{111}}]\label{3compact}
Let \( I \) be an open bounded interval of \( \mathbb{R} \), and let \( L\log L(I) \) denote the Zygmund space.
If we denote
\[
W^{1, L\log L}(I) = \{\, h \in L^1(I) \;\text{such that}\; \partial_x h \in L\log L(I) \,\},
\]
then the following embedding
\[
W^{1, L\log L}(I) \hookrightarrow C(I)
\]
is compact.
\end{lemma}
\begin{lemma}[Stability of maxima under uniform convergence {\cite[Lemma 2.2]{5}}]\label{3barles}
Let $\Omega \subset \mathbb{R}^n$ be an open set and let 
$(v_\varepsilon)_\varepsilon \subset C(\Omega)$ be a sequence of functions such that
\[
v_\varepsilon \to v \quad \text{in } C(\Omega).
\]
Assume that $x \in \Omega$ is a strict local maximum point of $v$. 
Then there exists a sequence $(x_\varepsilon)_\varepsilon \subset \Omega$ such that
\begin{itemize}
\item for every $\varepsilon$, the point $x_\varepsilon$ is a local maximum point of $v_\varepsilon$;
\item $x_\varepsilon \to x$ as $\varepsilon \to 0$;
\item $v_\varepsilon(x_\varepsilon) \to v(x)$ as $\varepsilon \to 0$.
\end{itemize}
\end{lemma}
\end{appendix}
\bibliographystyle{plain}
\bibliography{bibliography}

\end{document}